\documentclass[12pt,reqno,a4paper]{amsart}

\usepackage{amsmath,amsfonts,amsthm,amssymb,amsxtra,enumerate,mathtools}
\usepackage[normalem]{ulem}
\usepackage[utf8]{inputenc}
\usepackage[T1]{fontenc}
\usepackage{lmodern}
\usepackage{bbm}
\usepackage{mathrsfs}
\usepackage{enumerate}
\usepackage{amsfonts, commath, dsfont, setspace, import, xcolor, enumerate, enumitem, graphicx, comment, hyperref}
\usepackage{fancyhdr}

\renewcommand{\epsilon}{\varepsilon}

\renewcommand{\phi}{\varphi}

\let\thmref\ref
\renewcommand{\ref}[1]{(\thmref{#1})}

\newcommand{\eq}[1]{\begin{equation}#1\end{equation}}
\newcommand{\eqN}[1]{\begin{equation*}#1\end{equation*}}

\newcommand{\vertiii}[1]{{\vert\kern-0.25ex\vert\kern-0.25ex\vert #1\vert\kern-0.25ex\vert\kern-0.25ex\vert}}

\numberwithin{equation}{section}

\newtheorem*{thm*}{Theorem}
\newtheorem{thm}{Theorem}[section]
\newtheorem{propo}[thm]{Proposition}
\newtheorem{lemma}[thm]{Lemma}
\newtheorem{coro}[thm]{Corollary}

\theoremstyle{definition}

\newtheorem{defn}[thm]{Definition}
\newtheorem{exmp}[thm]{Example}

\theoremstyle{remark}

\newtheorem{remark}[thm]{Remark}

\title[Logarithmic Schr\"odinger equations in infinite dimensions]{Logarithmic Schr\"odinger equations in infinite dimensions}

	\author{Larry Read} 
	\address{Larry Read, Department of Mathematics, Imperial College London, London SW7 2AZ, UK}
	\email{l.read19@imperial.ac.uk}
	
	\author{Boguslaw Zegarlinski}
	\address{Boguslaw Zegarlinski, Institut de Mathématiques de Toulouse, UMR5219, CNRS}
	\email{boguslaw.zegarlinski@cnrs.fr}

    \author{Mengchun Zhang} 
	\address{Mengchun Zhang, Department of Mathematics, Imperial College London, London SW7 2AZ, UK}
	\email{mengchun.zhang15@imperial.ac.uk}
    \date{}

\begin{document}

\begin{abstract}
    We study the logarithmic Schr\"odinger equation with finite range potential on $\mathbb{R}^{\mathbb{Z}^d}$. Through a ground-state representation, we associate and construct a global Gibbs measure and show that it satisfies a logarithmic Sobolev inequality. We find estimates on the solutions in arbitrary dimension and prove the existence of weak solutions to the infinite-dimensional Cauchy problem.
\end{abstract}

\maketitle

\raggedbottom 


\section{Introduction}

Consider the logarithmic Schr\"odinger equation (LSE) given by

\eq{i \partial_t \psi = H \psi \equiv - \Delta \psi + V \psi + \lambda \psi \log
\frac{| \psi |^2}{\int | \psi |^2 d_{} \mathbf{x}} \label{LSEPsi}}
with $\lambda \in \mathbb{R}$ and where $\Delta$ denotes the Laplacian in
$\mathbb{R}^n$. 
It has the following property of statistical independence for non interacting systems. Suppose the potential can be written as $V(\mathbf{x}) = V_1 (\mathbf{x}_1) + V_2 (\mathbf{x}_2)$, where $V_j$ depends on
coordinates $\mathbf{x}_j$, $j = 1, 2$, such that $\{\mathbf{x}_1 \} \cap \{
\mathbf{x}_2\} = \varnothing$ and $\mathbf{x} = (\mathbf{x}_1, \mathbf{x}_2)
\in \mathbb{R}^{n}$. Let
\[H_j \psi_j \equiv - \Delta_j \psi _j+ V_j \psi_j+ \lambda \psi_j\log\frac{| \psi_j
|^2}{\int|\psi_j|^2d\mathbf{x}_j}\]
with $\Delta_j \equiv \Delta_{\{\mathbf{x}_j\}}$. Then, for an initial condition $\psi(\mathbf{x},0) = \psi_1(\mathbf{x}_1,0)\psi_2(\mathbf{x}_2,0)$, the function $\psi = \psi_1 \psi_2$ solves the equation \eqref{LSEPsi}
\[i \partial_t \psi = H \psi = (\psi_2 H_1 \psi_1 + \psi_1 H_2 \psi_2).\]
Thus the probability density $| \psi |^2$ of a composite system, consisting of two non-interacting subsystem,
is described by the product of probability densities $| \psi_j |^2, j = 1, 2$,
describing both systems separately. Non-interacting systems remain statistically independent. Furthermore, for the solutions of the above equations the normalisation $\int | \psi |^2d\mathbf{x} = 1$ is independent of time and we have the following physical property of additivity of energy
\[ \langle \psi, H \psi\rangle =\sum_j\langle \psi_j, H_j \psi_j\rangle.\]

To study the asymptotic stability of large interacting systems we investigate an analog of the ground state representation. To this end, for the equation \eqref{LSEPsi} we would like to consider the solution $\psi\equiv\phi e^{-\frac12U}$ with some $U:\mathbb R^n\to\mathbb R$ independent of time. By choosing 
\eq{-\frac{1}{2} \Delta U + \frac{1}{4} | \nabla U |^2 + \lambda U = V \label{UV}}
we can rewrite our equation as follows
\eq{i \partial_t \phi = -\mathcal L \phi + \lambda \phi \log \frac{| \phi |^2}{\mu | \phi 
|^2} \label{GSLSE}}
with the linear operator $\mathcal L\equiv\Delta-\nabla U\cdot\nabla$ and the measure $\mu (\phi) \equiv \int \phi e^{-U}d\mathbf x$.

Later on we examine (\ref{GSLSE}) in finite and infinite dimensional spaces, therefore it will be vital that the formal Gibbs measure $d\mu \equiv e^{-U}d\mathbf{x}$ can be well defined as a probability measure. In particular, to discuss the equation in an infinite dimensional setting, where the Lebesgue measure is not well defined, it will be necessary to have a reference probability measure. For this we will need to add some technical conditions. First of all we note that any potential energy function $V$  
is required to have an additive structure, which includes a local external potential as well as a multi-particle interaction. Since the relation between $V$ and $U$ is nonlinear, even an additive  structure of $U$ may not necessarily provide a stable $V$. However we notice that this can be achieved if $U$ itself is finite range. This provides us with an interesting family of systems we will consider in this paper; however more general cases should be discussed elsewhere.

The LSE plays an important role in the class of nonlinear Schr\"odinger equations. Indeed, the equations with logarithmic nonlinearity admit a wide range of applications, 
for example, in quantum mechanics \cite{BBM1,BBM2}, quantum optics  \cite{BSSSSCh}, modelling magma transport  \cite{DeMFGL}, nuclear physics \cite{H}, transport and diffusion
phenomena \cite{HAL}, 
 information theory \cite{Br,Y},
 quantum gravity \cite{Zl}, theory
of Bose–Einstein condensation \cite{AZ} and others.

Over a long period, including a number of very recent publications, the LSE was extensively studied in mathematical literature. 
In particular, many results on existence and smoothness of solutions to different forms of the LSE have been provided (see \cite{A1}, \cite{A2}, \cite{BCST},   \cite{CG}, \cite{CS}, \cite{C1}, \cite{CH1}, \cite{CH2}, \cite{DAMS}, 
\cite{F1},
\cite{F2}, \cite{GLN},  \cite{HMN},  \cite{T}). These works concern finite-dimensional configuration space, which rules out the systems with infinitely many particles. This paper aims to explore such systems by introducing an infinite-dimensional setup for the LSE in the context of statistical mechanics.

In Section \ref{2-Setup} we introduce the infinite-dimensional configuration space $\Omega\equiv\mathbb R^{\mathbb Z^d}$ and finite-range potential $U$. We present a construction of the associated global Gibbs measure $\mu$ and the operator $\mathcal L$ corresponding to the Dirichlet form in $ L^2(\mu)$. In Section \ref{LSI}, we show that under a set of well-chosen assumptions the measure $\mu$ satisfies a logarithmic Sobolev inequality.
\begin{thm} \label{muLSISimplified}
    The global Gibbs measure $\mu$ satisfies a log-Sobolev inequality on $\Omega=\mathbb{R}^{\mathbb{Z}^d}$ (under suitable growth conditions on $U$), that is, for all $\phi\in H^1(\mu)$ 
    \eq{
    {\rm Ent}_\mu\left(|\phi|^2\right)\equiv \mu\left(|\phi|^2\log\frac{|\phi|^2}{\mu|\phi|^2}\right) \leq c\|\nabla \phi\|_{L^2(\mu)}^2
    \label{eqn:muLSI}}
    with some constant $c\in(0,\infty)$ independent
    of the function $\phi$.
\end{thm}
In particular, it follows from \eqref{eqn:muLSI} that the right hand side of the LSE \eqref{GSLSE} defined by addition of two unbounded operators can be well defined, uniformly in dimension, for an appropriate choice of $\lambda$. Although the general idea of our proof follows those of \cite{GZ}, \cite{SZ1}, \cite{Z1}, \cite{Z2}, \cite{Z3}, \cite{SZ2} we need some care due to  non-compactness of the space $\Omega=\mathbb R^{\mathbb Z^d}$ and the unboundedness of the interaction potential. The full statement, after the constructions in Section \ref{2-Setup}, is given in Theorem \ref{muLSI}.

In Section \ref{Analyticity} we show that our strategy can be used to recover nice, direct results on the analyticity of the theory, providing an alternative point of view to the one presented in \cite{SZ1}.

Following the construction of $\mu$ and proof of LSI, we will investigate various estimates for the solutions of LSE in Section \ref{LSEBound}. We summarise these in the statement below.
\begin{thm}
     Assume that $U$ satisfies the same conditions in Theorem 1.1. Let $\phi$ be the solution to (\ref{GSLSE}) in arbitrary dimension with initial data $f\in H^1(\mu)$ dependent on a finite number of variables, i.e. $f$ is a function of $\{x_{k}\}_{k\in \Lambda}$, where $\Lambda\subset\subset \mathbb{Z}^d$. Then, for $\lambda>0$, it follows that
    \[\|\nabla\phi(\cdot,t)\|_{L^2(\mu)}^2\leq e^{2\lambda t}\|\nabla f\|_{L^2(\mu)}^2 \quad\text{ and }\quad  {\rm Ent}_\mu\left(|\phi(\cdot,t)|^2\right)\leq(\lambda^{-1}+c)\|\nabla f\|_{L^2(\mu)}^2\]
    for all $t>0$, where $c$ is the constant in \eqref{eqn:muLSI}. Futhermore, for any $j\in\mathbb{Z}^d$
    \[\norm{\nabla_{\{x_j\}}\phi(\cdot,t)}_{L^2(\mu)}^2\leq e^{-N_j}\|\nabla f\|_{L^2(\mu)}^2\]
    holds for all $t\leq \varepsilon N_j$, where $ \varepsilon= \varepsilon(\lambda,U)>0$ and $N_j$ is proportional to the distance between $j$ and $\Lambda$.
\end{thm}
The last bound tells us that, in finite time, the LSE can only introduce small dependence on variables far away from the dependence of the initial data. This property is analogous to the one discussed in the semigroup case (see e.g. \cite{SZ1}, \cite{Z3}) and is vital for the construction of the solution of the infinite dimensional LSE. For the full statements of these bounds, see theorems \ref{Conservation}, \ref{GradEntbound} and \ref{finiteprop} in Section \ref{LSEBound}.

In Section \ref{Convergence} we study the existence of weak solutions to the infinite-dimensional LSE. We restrict to the one-dimensional lattice and introduce the following result.

\begin{thm}
Let $U$ be a finite-range potential with bounded multi-spin interaction (see Section 
\ref{Convergence}) and consider initial data $f\in H^1(\mu)$ dependent on finitely many variables. Then there exists a weak solution to the infinite-dimensional LSE problem, i.e. there exists $\phi\in H^1(\mu)$ such that
\eqN{
\begin{split}
    &\mu\left(g i\partial_t \phi\right) = \mu\bigg(\nabla g\cdot\nabla\phi+g\lambda \phi\log\frac{|\phi|^2}{\mu|\phi|^2}\bigg)\\
    &\phi\vert_{t=0}=f
\end{split}
}
for all smooth compactly supported function $g:\mathbb R^{\mathbb{Z}}\to\mathbb C$ that depends on finitely many variables.
\end{thm}

For the proof we construct a sequence of finite sets $\Lambda_n\subset\subset\mathbb Z$ such that, as $\Lambda_n$ invades $\mathbb Z$, the solution $\phi_{\Lambda_n}$ to the finite-dimensional LSE on $\mathbb R^{\Lambda_n}$ converges to a weak solution $\phi$ of the infinite-dimensional LSE on $\mathbb R^{\mathbb{Z}}$. In higher dimensions the same argument can be used, but more care has to be taken with the growth of the interaction (see Remark \ref{rmk:1dLattice}). The full statement of the result is given by Theorem \ref{existence}.

We conclude in Section \ref{Soliton} by discussing the construction of solitary solutions in finite and infinite dimensions.

\section{Infinite dimensional setup}\label{2-Setup}
\subsection{Configuration space} Consider our physical space being modelled by a lattice $\mathbb Z^d$ ($d\in\mathbb N^+$), endowed with the $l^\infty$ distance $|i-j|\equiv\max_{1\leq n\leq d}|i_n-j_n|$ for all $i,\,j$ in $\mathbb Z^d$. For simplicity, we define
\[{\rm dist}(\Lambda,\Lambda')\equiv\inf_{i\in\Lambda, j\in\Lambda'}|i-j|,\quad{\rm diam}(\Lambda)\equiv\sup_{i,j\in\Lambda}|i-j|\quad\ \forall\,\Lambda,\Lambda'\subset\mathbb Z^d.\]
We assign each $j\in\mathbb Z^d$ a random variable $x_j$ with values in $\mathbb R$. Then the configuration of our system is given by $x=(x_j)_{j\in\mathbb Z^d}$, living in the infinite-dimensional space $\Omega\equiv\mathbb R^{\mathbb Z^d}$. 

Denoting by $\Lambda\subset\subset\mathbb Z^d$ any compact subset $\Lambda$ of $\mathbb Z^d$, we define the projection $\pi_\Lambda:\Omega\to\mathbb R^\Lambda$ by
\[\pi_\Lambda(x)=(x_j)_{j\in\Lambda}\equiv x_\Lambda.\]
The $\sigma$-algebra $\Sigma$ on $\Omega$ is the smallest $\sigma$-algebra that makes all projections measurable. Similarly, any sub $\sigma$-algebra $\Sigma_\Lambda$ on $\mathbb R^\Lambda$ is the smallest $\sigma$-algebra for which $\pi_{\{j\}}$ is measurable for all $j\in\Lambda$. 

We say that a function $f$ on $\Omega$ is localised in $\Lambda$ if it is $\Sigma_\Lambda$-measurable (or equivalently, $f$ only depends on coordinates in $\mathbb R^\Lambda$). Let $\Lambda_f\subset\mathbb Z^d$ be the smallest set in which $f$ is localised, we call $f$ a local or cylinder function if $\Lambda_f$ is compact. Denote by ${\mathfrak{F}(\Omega)}$ the set of all local and bounded functions $f:\Omega\to\mathbb C$.

\subsection{Interactions and Gibbs measures} The potential on $\Omega$ will be given by a family $\{J_X\}_{X\subset\subset\mathbb Z^d}$ of differentiable local functions with finite range $R\in\mathbb N^+$. That is for each $X\subset\subset\mathbb Z^d$, $J_X:\Omega\to\mathbb R$ is localised in $X$, twice differentiable (as a function on $\mathbb R^X$) and $J_X\equiv0$ if ${\rm diam}(X)>R$. The potential energy $U_\Lambda$ in any finite volume $\Lambda\subset\subset\mathbb Z^d$ is well-defined by
\[U_\Lambda(x)=\sum_{\Lambda\cap X\neq\varnothing}J_X(x).\]
\begin{exmp}[Bilinear Potential]
Consider the family of potentials
\[J_X(x)=\left\{\begin{tabular}{ c l }
$C_{ii}x_i^2$&\ \ {\rm if\ }$X=\{i\}$\\
$C_{ij}x_ix_j$&\ \ {\rm if\ }$X=\{i,j\}$\\
0&\ \ {\rm otherwise}  
\end{tabular}\right.\]
where $C_{ij}\in\mathbb R$ uniformly bounded and $C_{ij}=0$ if $|i-j|>R$. The associated $U_\Lambda$, the bilinear interaction in $\Lambda$, follows from the above. 
\end{exmp}
Note that in general $U_\Lambda$ may depend on $x_j$ for some $j\notin\Lambda$. Such dependence outside $\Lambda$ shall be fixed as boundary conditions, i.e. $x_{\Lambda^{\rm c}}=\omega_{\Lambda^{\rm c}}$ for some constant $\omega\in\Omega$. For this reason, we can write
\[U_\Lambda(x)=U_\Lambda^\omega(x)\equiv U_\Lambda(x_\Lambda|x_{\Lambda^{\rm c}}=\omega_{\Lambda^{\rm c}}).\]
In other words, $U_\Lambda^\omega$ represents the energy of configuration $x_\Lambda$ in the finite sub-system $\mathbb R^\Lambda$, 
with all variables outside this system being fixed by $\omega$. The (local) Gibbs measure on $\mathbb R^\Lambda$ with boundary conditions $\omega$ is, by definition, the probability measure $\mathbb E_\Lambda^\omega$ satisfying
\[d\mathbb E_\Lambda^{\omega}=\frac{e^{-U_\Lambda^\omega}}{\int_{\mathbb R^\Lambda}e^{-U_\Lambda^\omega}\,dx_\Lambda}dx_\Lambda \]
where $dx_\Lambda$ is the product Lebesgue measure on $\mathbb R^\Lambda$. For $\mathbb E_\Lambda^\omega$ to be well-defined in all finite volume $\Lambda$, we assume 
\eq{\qquad0<\int_{\mathbb R^\Lambda}e^{-U_\Lambda^\omega}\,dx_\Lambda<\infty\quad\forall\,\Lambda\subset\subset\mathbb Z^d\quad\forall\,\omega\in\Omega \label{Uassum}}
and hereinafter we only consider $U_\Lambda$ that satisfies the condition \eqref{Uassum}. 

The quantity $d\mathbb E_\Lambda^{\omega}$ has all the properties of a regular conditional probability of the system $\mathbb R^\Lambda$ being in configuration $x_\Lambda$, given the boundary $\omega_{\Lambda^{\rm c}}$ outside the system. Frequently it will be convenient to identify $\mathbb E_\Lambda^\omega$ as an operator defined by
\[\mathbb E_\Lambda^\omega f\equiv\int f(x_\Lambda|x_{\Lambda^{\rm c}}=\omega_{\Lambda^{\rm c}})\,d\mathbb E_\Lambda^{\omega} \]
on measurable functions $f:\Omega\to\mathbb C$. To ease the notation, we will write $\mathbb E_\Lambda=\mathbb E_\Lambda^\omega$ when the role of $\omega$ is insignificant, and we call the family $\{\mathbb E_\Lambda\}_{\Lambda\subset\subset\mathbb Z^d}$ a local specification since it specifies the probability density in every local system $\mathbb R^\Lambda$. In the following lemma we recall some useful properties of the local specification.

\begin{lemma}[\cite{BHK}] \label{localELambda}

Let $\{\mathbb E_\Lambda\}_{\Lambda\subset\subset\mathbb Z^d}$ be the local specification defined as above. Then the following properties are satisfied for all boundary conditions: 

\begin{enumerate}[label={\rm(\roman*)}]
    \item { Normalisation:} $\mathbb E_\Lambda(1)=1$.
    \item { Additivity:} Let $\{f_n\geq 0\}$ be a sequence in ${\mathfrak{F}(\Omega)}$, then
    \[\sum_n\mathbb E_\Lambda f_n=\mathbb E_\Lambda\bigg(\sum_nf_n\bigg).\]
    \item { Locality:} $\mathbb E_\Lambda f$ is $\Sigma_{\Lambda^c}$-measurable for $f\in {\mathfrak{F}(\Omega)}$, and if $f$ is $\Sigma_{\Lambda^c}$-measurable then $\mathbb E_\Lambda f=f$.
    \item { Compatibility:} If $\Lambda_1\subset\Lambda_2$ then
    \[\mathbb E_{\Lambda_2}\mathbb E_{\Lambda_1}=\mathbb E_{\Lambda_2}.\]
\end{enumerate}

\end{lemma}

Using the local specification, we can now define a global Gibbs measure $\mu$ on the whole space $\Omega$. 
Given any boundary condition $\omega\in\Omega$, the probability $\mu$ on $\Omega$ conditional to $\omega$ should match the probability kernel $\mathbb E_\Lambda^{\omega}$ on $\mathbb R^\Lambda$. Hence, we require
\eq{\mu(f\,|\,\{x_{\Lambda^{\rm c}}=\omega_{\Lambda^{\rm c}}\})=\mathbb E_\Lambda^\omega f \label{conditionalmu}}
for all $\Lambda\subset\subset\mathbb Z^d$, $\omega\in\mathbb R^{\Lambda^{\rm c}}$ and $f\in {\mathfrak{F}(\Omega)}$. The statement \eqref{conditionalmu} is equivalent to
\eq{\mu \mathbb E_\Lambda f=\mu f, \label{DLR}}
known as the DLR equation introduced by Dobrushin, Lanford, and Ruelle \cite{DLR2}.

Consequently, any solution to (\ref{DLR}) 
shall be called a (global) Gibbs measure associated to the local specification $\{\mathbb E_\Lambda\}_{\Lambda\subset\subset\mathbb Z^d}$. Later we formulate conditions that guarantee the existence and uniqueness of the global Gibbs measure $\mu$. For a general discussion of the solutions to the DLR equation see e.g. \cite{BHK}, \cite{DLR2}, \cite{DLR1}, \cite{DLR3}, \cite{FV}.

\subsection{Differential operators}
For simplicity, we shall denote by $\nabla_j\equiv\partial/\partial x_j$ and $\Delta_j\equiv(\partial/\partial x_j)^2$ the derivatives on $\mathbb R^{\{j\}}$. For any differentiable function $f:\Omega\to\mathbb C$ and $\Lambda\subset\subset\mathbb Z^d$, we define the gradient $\nabla_\Lambda f\equiv\left(\nabla_jf \right)$$_{j\in\Lambda}$ and $\nabla f\equiv\left( \nabla_jf\right)$$_{j\in\mathbb Z^d}$, with the Laplacian $\Delta_\Lambda\equiv\sum$$_{j\in\Lambda}\Delta_j$ and $\Delta\equiv\sum$$_{j\in\mathbb Z^d}\Delta_j$. The dot product is naturally defined by
\[\nabla f\cdot\nabla g\equiv\sum_{j\in\mathbb Z^d}(\nabla_jf)(\nabla_jg),\quad\nabla_\Lambda f\cdot\nabla_\Lambda g\equiv\sum_{j\in\Lambda}(\nabla_jf)(\nabla_jg),\]
inducing the $L^2$ length $|\nabla f|^2\equiv\nabla f\cdot\nabla f$ and $|\nabla_\Lambda f|^2\equiv\nabla_\Lambda f\cdot\nabla_\Lambda f$. Consider $f\in {\mathfrak{F}(\Omega)}$ that are twice differentiable, then the linear operator
$$\mathcal L_\Lambda f=\Delta_\Lambda f - \nabla_\Lambda U_\Lambda\cdot\nabla_\Lambda f$$
is well-defined for all $\Lambda\subset\subset\mathbb Z^d$.

\subsection{Properties of Gibbs measures}
We recall some well-known properties and coercive inequalities relating to the operator $\mathcal L_\Lambda=\Delta_\Lambda  - \nabla U_\Lambda\cdot\nabla_\Lambda$ and the associated Gibbs measures $\mathbb E_\Lambda$ and $\mu$ that we constructed in the previous subsections. First, we show that the Gibbs measures are invariant and reversible with respect to $\mathcal L_\Lambda$.

\begin{lemma}
    For any $\Lambda\subset\subset\mathbb Z^d$, $\mu$ and $\mathbb E_\Lambda$ are invariant with respect to $\mathcal L_\Lambda$, i.e. for all $\omega\in\Omega$ it holds
    \[\mu(\mathcal L_\Lambda f)=\mathbb E_\Lambda^\omega(\mathcal L_\Lambda f)=0. \]
\end{lemma}

\begin{proof}
Let $j\in\Lambda$ and we have \[\mathbb E_\Lambda(\Delta_jf)=\frac1{Z_\Lambda}\int(\Delta_jf)e^{-U_\Lambda}dx_\Lambda=\frac1{Z_\Lambda}\int\nabla_jf\cdot\nabla_j\big(e^{-U_\Lambda}\big)dx_\Lambda=\mathbb E_\Lambda(\nabla_jU_\Lambda\cdot\nabla_jf)\]
due to integration by parts. Then $\mathbb E_\Lambda(\mathcal L_\Lambda f)=0$ and by the DLR equation one obtains the same property for $\mu$.
\end{proof}

\begin{lemma}\label{Lem 2.3}
    For any $\Lambda\subset\subset\mathbb Z^d$, $\mu$ and $\mathbb E_\Lambda$ are reversible for $\mathcal L_\Lambda$, i.e. for all $\omega\in\Omega$ it holds
    \[\mu(f\mathcal L_\Lambda g)=\mu(g\mathcal L_\Lambda f),\quad \mathbb E_\Lambda^\omega(f\mathcal L_\Lambda g)=\mathbb E_\Lambda^\omega(g\mathcal L_\Lambda f).\]
\end{lemma} 

\begin{proof}
For any $j\in\Lambda$ we have
\begin{equation*}\begin{split}\int(f\Delta_jg-g\Delta_jf)e^{-U_\Lambda}dx_\Lambda&=\int\nabla_j(f\nabla_jg-g\nabla_jf)e^{-U_\Lambda}dx_\Lambda\\&=-\int(f\nabla_jg-g\nabla_jf)\nabla_j\big(e^{-U_\Lambda}\big)dx_\Lambda\end{split}\end{equation*}
which gives $\mathbb E_\Lambda(f\mathcal L_\Lambda g)=\mathbb E_\Lambda(g\mathcal L_\Lambda f)$. The DLR equation leads to the same reversibility of $\mu$.
\end{proof} 
\begin{defn}
    We say $\mathbb E_\Lambda$ satisfies the spectral gap inequality (SGI) with coefficient $c_{SG}\in(0,\infty)$ if
    \[\mathbb E_\Lambda^\omega|f-\mathbb E_\Lambda^\omega f|^2\leq \mathbb E_\Lambda^\omega(f^*(-\mathcal L_\Lambda)f)=c_{SG}\mathbb E_\Lambda^\omega|\nabla_\Lambda f|^2 \]
    and $\mathbb E_\Lambda$ satisfies the logarithmic Sobolev inequality (LSI) with coefficient $c_{LS}\in(0,\infty)$ if
    \[{\rm Ent}_\Lambda(|f|^2)\equiv \mathbb E_\Lambda^\omega\left(|f|^2\log\frac{|f|^2}{\mathbb E_\Lambda^\omega|f|^2}\right)\leq c_{LS} \mathbb E_\Lambda^\omega|\nabla_\Lambda f|^2 \]
    for all boundary conditions $\omega\in\Omega$.
\end{defn} 
%
\begin{remark}  Note that LSI for real functions implies similar inequality for complex functions.
To see that, let $f_R\equiv{\rm Re}(f)$, $f_I\equiv{\rm Im}(f)$ and suppose that $\mathbb E_\Lambda |f|^2=\mathbb E_\Lambda |f_R|^2+\mathbb E_\Lambda |f_I|^2 =1$.  Then by the convexity of the function $x\log x$ we have
\begin{equation*}\begin{split}
{\rm Ent}_\Lambda(|f|^2 ) &= {\rm Ent}_\Lambda\left( \mathbb E_\Lambda |f_R|^2 \frac{|f_R|^2}{\mathbb E_\Lambda |f_R|^2}+  \mathbb E_\Lambda |f_I|^2 \frac{| f_I|^2}{\mathbb E_\Lambda |f_I|^2} \right) \\
&\leq {\rm Ent}_\Lambda\left( \frac{|f_R|^2}{\mathbb E_\Lambda |f_R|^2}\right)\mathbb E_\Lambda |f_R|^2 + {\rm Ent}_\Lambda\left( \frac{|f_I|^2}{\mathbb E_\Lambda |f_I|^2} \right)\mathbb E_\Lambda |f_I|^2.
\end{split}
 \end{equation*}
Hence applying  LSI with the real functions $|f_R|/(\mathbb E_\Lambda |f_R|^2)^\frac12$ and $|f_I|/(\mathbb E_\Lambda |f_I|^2)^\frac12)$ and using inequality $|\nabla |\phi||\leq |\nabla \phi|$ , we get desired LSI for complex functions $f$ with the same coefficient $c$.
\end{remark}

\begin{lemma} \label{LSItoSGI}
    If\ \ $\mathbb E_\Lambda$ satisfies the LSI with coefficient $c_{LS}$, then $\mathbb E_\Lambda$ satisfies the spectral gap inequality with coefficient $c_{SG}$ such that 
    \[c_{SG}\leq\frac12c_{LS}.\]
\end{lemma} 
\begin{proof} 
Let $f$ be a real function with zero expectation $\mathbb E_\Lambda f=0$. Let $\varepsilon>0$ be small enough, then by the LSI for $\mathbb E_\Lambda$, we have
\[\mathbb E_\Lambda\left(|1+\varepsilon f|^2\log\frac{|1+\varepsilon f|^2}{\mathbb E_\Lambda|1+\varepsilon f|^2}\right)\leq c\,\varepsilon^2\mathbb E_\Lambda|\nabla_\Lambda f|^2.\]
Since $\mathbb E_\Lambda f=0$, the Taylor expansion of the left hand side is $2\varepsilon^2\mathbb E_\Lambda|f|^2+\mathcal O(\varepsilon^3)$, which gives
\[\mathbb E_\Lambda|f|^2\leq \frac c2\mathbb E_\Lambda|\nabla_\Lambda f|^2.\]
Replace $f$ with $f-\mathbb E_\Lambda f$ and we obtain Lemma \ref{LSItoSGI} for real $f$. Finally, note that the spectral gap inequality for real functions $f$ implies corresponding  property for complex functions.
\end{proof} 
\begin{remark} The above property and proof is due to Rothaus  \cite{Rothaus1}. 
Another idea (due to S.G.Bokov and F.G\"otze), is to notice that right hand side of LSI is invariant with respect to a shift by a constant  and that one has
\[\frac12 \mathbb E_\Lambda|f-\mathbb E_\Lambda f|^2\leq \sup_{a\in\mathbb{R}} 
{\mathbb E_\Lambda\left(|f|^2\log\frac{|f|^2}{\mathbb E_\Lambda|f|^2}\right) }. \]
For an interesting discussion when $c_{SG} =\frac12c_{LS}$ is or is not true see \cite{Rothaus2}.
\end{remark}
Finally, we state an important, well-known, property of log-Sobolev inequalities.
\begin{lemma} \label{LSISGItensorise}
    If $\mu_1,\,\mu_2$ are two probability measures on two non-interacting systems $\Omega_1,\,\Omega_2$ satisfying LSI with coefficient $c_{LS}$ (or SGI with coefficient $c_{SG})$, then their tensorisation $\mu_1\otimes\mu_2$ on the product space $\Omega_1\times\Omega_2$ also satisfies a LSI with coefficient $c_{LS}$ (or an SGI with coefficient $c_{SG}$).  
\end{lemma} 
The proof can be found in Theorem 2.5 and Theorem 4.4 in \cite{GZ}.

\section{Logarithmic Sobolev Inequalities for Global Gibbs Measure \label{LSI}}
Hereafter, we make the assumption that the global Gibbs measure, $\mu$ exists, which we justify with Remark \ref{muexist} below. Following the construction in Section \ref{2-Setup}, we present sufficient conditions on the potential $U$ such that the associated  measure $\mu$ satisfies a logarithmic Sobolev inequality. We now restate Theorem \ref{muLSISimplified} in its full form.

\begin{thm}\label{muLSI}
Suppose $U_\Lambda$ satisfies the condition \eqref{Uassum} introduced in Section \ref{2-Setup} and $\|\nabla_i\nabla_jU_\Lambda\|_\infty\leq A$ for all $\Lambda\subset\subset\mathbb Z^d$, $i,j\in\mathbb Z^d$ with $i\neq j$ where $A>0$ is independent of $i,j,\Lambda$. If the finite-volume log-Sobolev inequality
\eq{\mathbb E_\Lambda\left(|f|^2\log\frac{|f|^2}{\mathbb E_\Lambda|f|^2}\right)\leq\tilde c\,\mathbb E_\Lambda|\nabla_\Lambda f|^2 \label{LambdaLSI}}
holds for all $\Lambda\subset\subset\mathbb Z^d$ and some $\tilde c>0$ independent of $\Lambda$, then the global Gibbs measure $\mu$ associated to $\{\mathbb E_\Lambda\}$ is unique and satisfies the log-Sobolev inequality
\eq{\mu\left(|f|^2\log\frac{|f|^2}{\mu|f|^2}\right)\leq c\mu|\nabla f|^2 \label{LSImu}}
for all $f\in H^1(\mu)$ and some constant $c\in(0,\infty)$ independent of $f$.
\end{thm}

\begin{remark} \label{rmk:BE}
A sufficient condition for the volume-uniform LSI \eqref{LambdaLSI} is that
\eq{\inf_{x\in\Omega}\Delta_iU_\Lambda(x)\geq B \label{StrongConvexity}}
for all $\Lambda\subset\subset\mathbb Z^d$, $i\in\Lambda$ and some $B>0$ independent of $i$ and $\Lambda$. This is a consequence of the Bakry-Emery criterion \cite{BE}, which states that if 
\eq{\mathbf{\Gamma}_2(f,f)\geq b\mathbf{\Gamma}_1(f,f) \label{BECondition}}
for some $b>0$ independent of $f$, where
\begin{align*}
    &\mathbf{\Gamma}_1(f,g)\equiv \mathcal L_\Lambda(fg)-f\mathcal L_\Lambda g-g\mathcal L_\Lambda f,\\
    &\mathbf{\Gamma}_2(f,g)\equiv \mathcal L_\Lambda(\mathbf{\Gamma}_1(f,g))-\mathbf{\Gamma}_1(f,\mathcal L_\Lambda g)-\mathbf{\Gamma}_1(g,\mathcal L_\Lambda f),
\end{align*}
then \eqref{LambdaLSI} holds with $\tilde c=2/b$. By direct calculation, \eqref{BECondition} is satisfied due to \eqref{StrongConvexity}. In other words, one has a volume-uniform LSI if the potential is strongly convex at all sites with the same strong convexity constant. 
\end{remark}
For application of the above result, we return to the examples of the bilinear potential and its perturbation by bounded local interactions.
\begin{exmp} \label{Exmp:bilinear}
Consider the bilinear potential \[U_\Lambda(x)=\sum_{\{i,j\}\cap\Lambda\neq\varnothing}C_{ij}x_ix_j\] 
where $C^-\leq C_{ii}\leq C^+$ and $|C_{ij}|\leq C^+$ for some $0<C^-\leq C^+$, and $C_{ij}=0$ if $|i-j|>R$. Then $U_\Lambda$ satisfies all conditions in Theorem \ref{muLSI} by Remark \ref{rmk:BE}.

\end{exmp}

\begin{exmp} \label{Exmp:bilinearPerturbed}
Consider the same bilinear potential but with bounded perturbations \[U_\Lambda(x)=\sum_{\{i,j\}\cap\Lambda\neq\varnothing}C_{ij}x_ix_j+\varepsilon\sum_{\substack{X\subset\subset\mathbb Z^d\\X\cap\Lambda\neq\varnothing}}W_X(x_X),\] 
where $\varepsilon\in\mathbb R$, $W_X\in C^2(\mathbb R^X)$ has uniformly in $X$ bounded derivatives of order $n=0,1,2$ and $W_X\equiv 0$ if diam$(X)>R$. Suppose $C_{ij}$ satisfies the same conditions as above, then for sufficiently small $|\varepsilon|$ all conditions in Theorem \ref{muLSI} are satisfied.

\end{exmp}
As a final example we introduce the restriction to the interactions with no boundary conditions. Later on we will use this extensively to study existence results for the LSE.
\begin{exmp}
\label{mucirc}
Theorem \ref{muLSI} holds true when the boundary conditions are removed. That is, suppose $U$ satisfies the conditions of Theorem \ref{muLSI}, then we denote by 
\eqN{U_\Lambda^\circ(x)\equiv\sum_{X\subseteq\Lambda}J_X(x)} 
the restriction to $J_X$ strictly localised inside $\Lambda$. Then, we denote by $\mathbb E_\Lambda^\circ$ the local Gibbs measure corresponding to $U_\Lambda^\circ$, i.e. $d\mathbb E_\Lambda^\circ\equiv(e^{-U_\Lambda^\circ}dx_\Lambda)(\int\!e^{-U_\Lambda^\circ}dx_\Lambda)^{-1}$, and by $\mu^\circ$ the global Gibbs measure associated to $\{\mathbb E_\Lambda^\circ\}$ through the DLR equation. Then Theorem \ref{muLSI} holds with $(U_\Lambda,\mathbb E_\Lambda,\mu)$ replaced by $(U_\Lambda^\circ,\mathbb E_\Lambda^\circ,\mu^\circ)$.
\end{exmp}

\begin{remark}[Existence]
We remark that by the compactness argument in \cite{BHK}, the existence of $\mu$ in our setting can always be obtained by restricting the configuration space $\Omega$, to the space of slowly increasing sequences (tempered sequences) $\mathcal S_\Omega\equiv\big\{x\in\Omega:\exists N>0\text{ such that }\sup_j|x_j|/|j|^N<\infty\big\}$, provided the multi-spin interaction satisfies suitable growth conditions. The restriction to the subspace $\mathcal S_\Omega$ can be justified by the fact that $\mathcal S_\Omega$ is dense in $\Omega$, i.e. $\mu(\mathcal S_\Omega)=1$ (see \cite[Proposition~A.1]{BHK}), and that one can obtain compactness argument on such space.
\label{muexist}
\end{remark}

\subsection{Strategy to prove Theorem \ref{muLSI}}
The proof of Theorem \ref{muLSI} follows the ideas of \cite{GZ}, \cite{SZ1}, \cite{Z1}, \cite{Z2}, \cite{Z3}, \cite{SZ2}  with some modifications for the non-compactness of the space $\Omega\equiv\mathbb R^{\mathbb Z^d}$ and the unboundedness of the potential $U$. The strategy is to construct a probability measure $\Pi$ on $\Omega$ that satisfies a LSI and where $\Pi^n\to\mu$ as $n\to\infty$. Then one can obtain the $\mu$-LSI by using the $\Pi$-LSI repeatedly on a telescoping series of entropy terms associated to the sequence $\{\Pi^n\}_{n\in\mathbb N}$.

Here we present a construction of such $\Pi$, the idea is to define $\Pi$ as an infinite tensorisation of local Gibbs measures on equally sized cubes that partition the whole lattice. To this end, for some parameter $L\in\mathbb N$ and the interaction range $R\in\mathbb N$, we define the base cube $X_0\equiv [0,2(L+R)]^d\cap\mathbb Z^d$ and denote by $X_k\equiv k+X_0$ the translation of $X_0$ by the vector $k\in\mathbb Z^d$. For any $s\in\mathbb N$ let $v_s=(v_s^{(1)},...,v_s^{(d)})\in\{0,1\}^d$ be the binary representation of $s$, i.e. $s=\sum_{n=1}^dv_s^{(n)}2^{n-1}$. We define the following collection of cubes
\eq{\Gamma_s=\bigcup_{k\in T_s}X_k\label{Gammas}}
for $s=0,1,...,2^d-1$ and the translation set $T_s\equiv\big\{k\in\mathbb Z^d:k\in2(L+2R)\mathbb Z^d+(L+2R)v_s\big\}$.
By such construction, each $\Gamma_s$ contains disjoint cubes of shape $X_0$, equally distributed in $\mathbb Z^d$, and separated by a distance greater than $2R$. In addition, as $s$ goes over $\{0,1,...,2^d-1\}$ the vector $v_s$ will cover all directions in $\{0,1\}^d$, hence the union $\bigcup_{s=0,1,...,2^d-1}\Gamma_s=\mathbb Z^d$ covers the whole lattice. Let $\mathbb E_{X_k}$ be the local Gibbs measure on the cube $X_k$ and we set $\mathbb E_s\equiv\bigotimes_{X_k\subset\Gamma_s}\mathbb E_{X_k}$. Then we define the measure $\Pi$ by
\[\Pi^\omega\equiv\mathbb E_{2^d-1}^\omega \mathbb E_{2^d-2}\cdots\mathbb E_1\mathbb E_0 \]  with boundary conditions $x_{(\Gamma_{2^d-1})^{\mathsf c}}=\omega_{(\Gamma_{2^d-1})^{\mathsf c}}$. The measure $\Pi$ satisfies the following four properties.

\begin{propo} \label{TechnicalProp}
Let $\Pi$ be as constructed above, then the following hold. 
\begin{enumerate}[label={\rm(\roman*)}]
    \item The DLR equation holds:
    \[\mu\Pi f=\mu f.\]
    
    \item There exists $\bar c>0$ such that 
    \[\mu\Big(\Pi\big(|f|^2\log|f|^2\big)-\Pi|f|^2\log\Pi|f|^2\Big)\leq\bar c\,\mu|\nabla f|^2.\]
    
    \item There exists $\gamma\in(0,1)$ such that
    \[\mu\big|\nabla(\Pi |f|^2)^\frac12\big|^2\leq\gamma\mu|\nabla f|^2.\]
    
    \item The sequence $\Pi^n$ converges to $\mu$ almost surely, i.e.
    \[\qquad\qquad\qquad\lim_{n\to\infty}\Pi^n(f)=\mu(f)\qquad\ \ \mu\text{\,-\,a.s},\]
    and the limit defines a unique global Gibbs measure $\mu$ associated to $\{\mathbb E_\Lambda\}$.
\end{enumerate}
\end{propo}

Before proving Proposition \ref{TechnicalProp}, we first demonstrate how to derive the $\mu$-LSI \eqref{LSImu} by using $\Pi$.\vspace{.2cm}

\noindent
\begin{proof}[Proof of Theorem \ref{muLSI}]
By (i), we have
\[\mu\left(|f|^2\log\frac{|f|^2}{\mu |f|^2}\right)=\sum_{n=0}^{N-1}\mu\Big(\Pi\big(|f_n|^2\log|f_n|^2\big)-\Pi|f_n|^2\log\Pi |f_n|^2\Big)+\mu\left(|f_N|^2\log\frac{|f_N|^2}{\mu |f|^2}\right)\]
then by (ii), it follows
\[\mu\left(|f|^2\log\frac{|f|^2}{\mu |f|^2}\right)\leq\bar c\sum_{n=0}^{N-1}\mu|\nabla f_n|^2+\mu\left(|f_N|^2\log\frac{|f_N|^2}{\mu |f|^2}\right).\]
Since $\gamma\in(0,1)$, applying (iii) gives
\[\mu\left(|f|^2\log\frac{|f|^2}{\mu |f|^2}\right)\leq\frac{\bar c}{1-\gamma}\mu|\nabla f|^2+\mu\left(|f_N|^2\log\frac{|f_N|^2}{\mu |f|^2}\right).\]
Then by (iv) and the monotone convergence theorem, the last term follows
\[\mu\left(|f_N|^2\log\frac{|f_N|^2}{\mu |f|^2}\right)\to0\]
as $N\to\infty$, which gives the desired LSI for $\mu$ with coefficient $c\equiv\bar c/(1-\gamma)$.\vspace{.2cm}
\end{proof}

\begin{remark}
The choice of $\Pi$ is not unique.  In the proof Proposition \ref{TechnicalProp} the base cube $X_0$ can in fact take any shape as long as it is optimised under a set of conditions. Our choice of homogeneous cube partition is just to provide an easy presentation of the formulae.
\end{remark}

\subsection{Proof of Proposition \ref{TechnicalProp}}

Here we present a proof of conditions (i) to (iv). Condition (i) easily follows from the DLR equation \eqref{DLR} for each local Gibbs measure $\mathbb E_{X_k}$ in the cube setup. To verify conditions (ii) to (iv), we use the following lemma.
\begin{lemma} \label{LSIsweepout}
Suppose all conditions in Theorem \ref{muLSI} are satisfied, then for any cube $X_k$ and $j\notin{X_k}$ the `sweeping out' inequalities
\eq{\mu\Big|\nabla_j(\mathbb E_{X_k}f)\big|^2\leq \sum_{i\in\mathbb Z^d}\alpha_{ji}\mu|\nabla_if|^2 \label{sweepout2nd}}
and
\eq{\mu\Big|\nabla_j\big(\mathbb E_{X_k} |f|^2\big)^\frac12\big|^2\leq \sum_{i\in\mathbb Z^d}\alpha_{ji}\mu|\nabla_if|^2 \label{sweepout}}
hold with $0\leq\alpha_{ji}\leq D|X_0|e^{-M|j-i|}$ for some positive constants $D,M$ independent of $j$ and $X_k$.
\end{lemma}
Assuming Lemma \ref{LSIsweepout} is true, we now prove condition (ii). Denote by $f_{-1}\equiv f$ and $f_s\equiv(\mathbb E_s\mathbb E_{s-1}\cdots\mathbb E_0|f|^2)^{1/2}$, one has the telescopic representation
\[\mu\Big(\Pi\big(|f|^2\log|f|^2\big)-\Pi|f|^2\log\Pi|f|^2\Big)=\sum_{s=0}^{2^d-1}\mu \mathbb E_s\left(f_{s-1}^2\log\frac{f_{s-1}^2}{\mathbb E_sf_{s-1}^2}\right)\]
by the DLR equation. For any $s$, each local measure $\mathbb E_{X_k}$ with $X_k\subset\Gamma_s$ satisfies a LSI with coefficient $\tilde c$ by \eqref{LambdaLSI}, their tensorisation $\mathbb E_s$ therefore satisfies the same LSI by Lemma \ref{LSISGItensorise}, leading to 
\eq{\mu\Big(\Pi\big(|f|^2\log|f|^2\big)-\Pi|f|^2\log\Pi|f|^2\Big)\leq\tilde c\sum_{s=0}^{2^d-1}\mu|\nabla_{\Gamma_s}f_{s-1}|^2. \label{LSIforPi}}
We now present a simple study of the quantity $\nabla_{\Gamma_s}f_{s-1}$. For any $j\in\Gamma_s$, if $j\in\Gamma_{s-1}$ then $\nabla_jf_{s-1}=0$. If $j\notin\Gamma_{s-1}$, there exists only one cube $X_{k(j)}\subset\Gamma_{s-1}$ such that ${\rm dist}(j,X_{k(j)})\leq R$. In the second case, we have
\[\nabla_jf_{s-1}=\nabla_j(\mathbb E_{\Gamma_{s-1}\setminus X_{k(j)}}\mathbb E_{X_{k(j)}}f_{s-2}^2)^\frac12\leq\mathbb E_{\Gamma_{s-1}\setminus X_{k(j)}}\Big|\nabla_j(\mathbb E_{X_{k(j)}}f_{s-2}^2)^\frac12\Big|^2\]
by Cauchy-Schwartz inequality. Integrating this with $\mu$ and using the sweeping out inequality \eqref{sweepout} for $\Lambda= X(j)$, it follows
\[\mu|\nabla_jf_{s-1}|^2\leq \sum_{i\in\mathbb Z^d\setminus\Gamma_{s-2}}\alpha_{ji}\mu\abs{\nabla_{i}f_{s-2}}^2,\]
where we have $i\in\mathbb Z^d\setminus\Gamma_{s-2}$ because $f_{s-2}=(\mathbb E_{\Gamma_{s-2}}|f_{s-3}|^2)^{1/2}$ has no dependence in $\Gamma_{s-2}$. Now that $\nabla_jf_{s-1}$ is bounded by $\nabla_if_{s-2}$, the same estimate can be used recursively until we reach the gradients on $f$, i.e.
\eq{\mu|\nabla_jf_{s-1}|^2\leq\sum_{i\in\mathbb Z^d}\eta^{(s)}_{ji}\mu\abs{\nabla_if}^2,  \label{sweepiterate}}
where 
\eq{\eta^{(s)}_{ji}\equiv\sum_{i_1\in\mathbb Z^d\setminus\Gamma_{s-2}}\,\sum_{i_2\in\mathbb Z^d\setminus\Gamma_{s-3}}\cdots\sum_{i_{s-1}\in\mathbb Z^d\setminus\Gamma_0}\big(\alpha_{ji_1}\alpha_{i_1i_2}\cdots\alpha_{i_{s-1}i}\big). \label{etaji}}
Summing \eqref{sweepiterate} over all $j\in\Gamma_s$, the constant before $\mu|\nabla_if|^2$ is therefore
\eq{\sum_{j\in\Gamma_s}\eta^{(s)}_{ji}\leq\bigg(\sup_{i\in\mathbb Z^d}\sum_{k\in\mathbb Z^d}\alpha_{ki}\bigg)^s\leq (D|X_0|)^s\bigg(\sum_{k\in\mathbb Z^d}e^{-M|k|}\bigg)^s\leq (D|X_0|C_d)^s \label{etasum}}
for some finite constant $C_d=\sum_ke^{-M|k|}$ that only depends on $M$ and the lattice dimension $d$. Taking this back into \eqref{LSIforPi}, one gets 
\eq{\mu\Big(\Pi\big(|f|^2\log|f|^2\big)-\Pi|f|^2\log\Pi|f|^2\Big)\leq\bar c\mu|\nabla f|^2 \label{(III)Proof}}
with $\bar c\equiv\tilde c\big(1+(D|X_0|C_d)+...+(D|X_0|C_d)^{2^d}\big)$, which concludes the proof of (ii).\vspace{.2cm}

To verify (iii), we use the same estimate as in \eqref{sweepiterate} with $s=2^d$, which gives
\eq{\mu\big|\nabla(\Pi |f|^2)^\frac12\big|^2=\mu|\nabla_jf_{2^d-1}|^2\leq\sum_{i\in\mathbb Z^d}\eta^{(2^d)}_{ji}\mu|\nabla_if|^2. \label{eqn:PiGradBound1}}
By the cube construction \eqref{Gammas}, the intersection $X_k\cap X_{k'}$ between any two overlapping cubes $X_k$ and $X_{k'}$ is a (hyper)rectangle whose shortest edge has a length of $L$. Therefore, along each path $\{j,i_1,i_2,...,i_{2^d-1},i\}$ in the sum \eqref{etaji} with $s=2^d$, there exists at least one pair of adjacent sites $(i_n, i_{n+1})$ such that $|i_n-i_{n+1}|\geq L/2$. Incorporate this property into the estimate \eqref{etasum}, one gets
\eq{\mu\big|\nabla(\Pi |f|^2)^\frac12\big|^2\leq(D|X_0|C_d)^{2^d}e^{-L/2}\mu|\nabla f|^2. \label{eqn:PiGradBound2}}
Since $|X_0|$ grows polynomially with $L$, one can make $L$ sufficiently large so that $\gamma\equiv(D|X_0|C_d)^{2^d}e^{-L/2}<1$ and this concludes (iii).\vspace{.2cm}

To verify (iv), we redefine $f_{-1}\equiv f$ and $f_s\equiv\mathbb E_s\mathbb E_{s-1}\cdots\mathbb E_0(f)$. By the argument for \eqref{LSIforPi}, each $\mathbb E_s$ satisfies a LSI (and thus a SGI by Lemma \ref{LSItoSGI}) with coefficient $\tilde c$. Hence, one can bound the $\Pi$-variance by
\[\mu|f-\Pi f|^2\leq2^{2^d}\sum_{s=0}^{2^d-1}\mu\mathbb E_s|f_{s-1}-\mathbb E_sf_{s-1}|^2\leq2^{2^d}\tilde c\sum_{s=0}^{2^d-1}\mu|\nabla_{\Gamma_s}f_{s-1}|^2. \]
Comparing this with \eqref{LSIforPi}, the situation is essentially the same and therefore we can follow the steps from \eqref{LSIforPi} until \eqref{(III)Proof} to conclude that
\[\mu\Pi|f-\Pi f|^2\leq2^{2^d}\bar c\mu|\nabla f|^2\]
for the same $\bar c$ in \eqref{(III)Proof}. Similarly, the argument of \eqref{eqn:PiGradBound1} and \eqref{eqn:PiGradBound2} leads to the gradient bound
\[\mu|\nabla(\Pi f)|^2\leq\beta\mu|\nabla f|^2\]
for some $\beta\in(0,1)$. Using the above two estimates with $f$ replaced by $\Pi^nf$, we obtain that
\[\mu|\Pi^nf-\Pi^{n+1}f|^2\leq2^{2^d}\bar c\mu|\nabla(\Pi^nf)|^2\leq2^{2^d}\beta^n\bar c\mu|\nabla f|^2.\]
Since $\beta\in(0,1)$ and $\mu|\nabla f|^2<\infty$ for $f\in H^1(\mu)$, we can use Borel-Cantelli lemma to show that the sequence $\{\Pi^nf\}_n$ converges $\mu$-almost surely to some function $F_{\rm lim}:\Omega\to\mathbb C$ (for the technical details, see \cite[Lemma~5.4.6]{Ing}). Similarly, one can prove that $|\nabla(\Pi^nf)|$ converges to zero $\mu$-a.s., meaning that the limit $F_{\rm lim}$ must be a constant function. Therefore, we have 
\[\qquad\qquad\qquad\lim_{n\to\infty}\Pi^nf=F_{\rm lim}=\mu(F_{\rm lim})=\lim_{n\to\infty}\mu(\Pi^nf)=\mu(f)\qquad\mu\text{\,-\,a.s.}\]
where the limit can be taken out of $\mu$ by dominated convergence theorem (since $f$ is bounded), and the last equality follows from the DLR equation (i).

\subsection{Proof of Lemma \ref{LSIsweepout}}
We will first prove \eqref{sweepout} and then \eqref{sweepout2nd} follows similarly. Let $\Lambda$ be any cube $X_k$ in our setup (thus ${\rm diam(\Lambda)>2R}$). For any $j\notin\Lambda$ let $i\in\Lambda$ be such that $|j-i|>R$ and we denote by $F\equiv(\mathbb E_{\Lambda\setminus i}|f|^2)^{1/2}$. By the compatibility property $\mathbb E_\Lambda=\mathbb E_\Lambda\mathbb E_{\Lambda\setminus i}$ in Lemma \ref{localELambda}, we have
\eq{\Big|\nabla_j\big(\mathbb E_\Lambda |f|^2\big)^\frac12\Big|^2=\Big|\nabla_j\big(\mathbb E_\Lambda F^2\big)^\frac12\Big|^2\leq\frac{\big|\mathbb E_\Lambda(\nabla_j F^2)\big|^2}{4\mathbb E_\Lambda F^2}+\frac{\big|\mathbb E_\Lambda (F^2; -\nabla_j U_\Lambda)\big|^2}{4\mathbb E_\Lambda F^2} \label{gradjEf2}}
with the covariance $\mathbb E_\Lambda(f;g)\equiv \mathbb E_\Lambda(fg)-(\mathbb E_\Lambda f)(\mathbb E_\Lambda g)$. Let $I_1$ and $I_2$ be the first and second term on the right hand side of \eqref{gradjEf2}, then it simply follows $I_1\leq\mathbb E_\Lambda|\nabla_jF|^2$
by Cauchy-Schwartz inequality. For $I_2$, since $F$ has no dependence in $\Lambda\setminus i$ the covariance follows $\mathbb E_\Lambda (F^2; -\nabla_j U_\Lambda)=\mathbb E_\Lambda (F^2;\mathcal U)$ with $\mathcal U\equiv\mathbb E_{\Lambda\setminus i}(-\nabla_jU_\Lambda)$. To continue the calculation, we define $\tilde{\mathbb E}_\Lambda\equiv\mathbb E_\Lambda(d\tilde x_\Lambda)$ as an isomorphic copy of $\mathbb E_\Lambda(dx_\Lambda)$ on a variable $\tilde x_\Lambda$ independent of $x_\Lambda$, and similarly $\tilde F\equiv F(\tilde x_\Lambda)$, $\tilde{\mathcal U}\equiv\mathcal U(\tilde x_\Lambda)$, while all configurations outside $\Lambda$ are fixed by the same boundary conditions $\omega_{\Lambda^{\mathsf c}}$. Now we can rewrite the covariance as
\begin{align}
    \big|\mathbb E_\Lambda (F^2;\mathcal U)\big|^2&=\frac14\Big|\mathbb E_\Lambda\tilde{\mathbb E}_\Lambda\Big(\big(F^2-\tilde F^2\big)(\mathcal U-\tilde{\mathcal U})\Big)\Big|^2 \nonumber\\
    &\leq\frac14\Big(\mathbb E_\Lambda\tilde{\mathbb E}_\Lambda(F+\tilde F)^2\Big)\mathbb E_\Lambda\tilde{\mathbb E}_\Lambda\Big((F-\tilde F)^2(\mathcal U-\tilde{\mathcal U})^2\Big) \nonumber\\
    &\leq\Big(\mathbb E_\Lambda F^2\Big)\mathbb E_\Lambda\tilde{\mathbb E}_\Lambda\Big((F-\tilde F)^2(\mathcal U-\tilde{\mathcal U})^2\Big). \label{FUCov}
\end{align}
Combining $I_1$ and $I_2$, we obtain the bound
\eq{\Big|\nabla_j\big(\mathbb E_\Lambda |f|^2\big)^\frac12\Big|^2\leq\mathbb E_\Lambda|\nabla_jF|^2+\frac14\mathbb E_\Lambda\tilde{\mathbb E}_\Lambda\Big((F-\tilde F)^2(\mathcal U-\tilde{\mathcal U})^2\Big). \label{eqn:nablajEf}}
Note that $\mathcal U$ only has one free variable on $x_i$ since all configurations are integrated by $\mathbb E_{\Lambda\setminus i}$ inside $\Lambda\setminus i$ and fixed by $\omega$ outside $\Lambda$. Therefore, we can write $\mathcal U(x_\Lambda)=\mathcal U(x_i)$ and estimate the difference
\[|\mathcal U(x_i)-\mathcal U(\tilde x_i)|=\bigg|\int_{\tilde x_i}^{x_i}\nabla_i\,\mathcal U(\hat x_i)\,d\hat x_i\bigg|\leq|x_i-\tilde x_i|\|\nabla_i\,\mathcal U\|_\infty=|x_i-\tilde x_i|\big\|\mathbb E_{\Lambda\setminus i}(\nabla_iU_{\Lambda\setminus i};\nabla_jU_{\Lambda\setminus i})\big\|_\infty \]
where the last step follows from the fact $\nabla_jU_\Lambda=\nabla_jU_{\Lambda\setminus i}$ since $|i-j|>R$. 

To bound the above covariance, we use the following result adapted from \cite[Theorem~2.4 and Remark~2.6]{HM}. 
\begin{thm} \label{thm:Mixing}
If $\|\nabla_i\nabla_jU_\Lambda\|_\infty\leq Ae^{-|j-i|}$ for some $A>0$ independent of $i,j,\Lambda$, then the volume-uniform LSI \eqref{LambdaLSI} is equivalent to the following mixing condition
\eq{|\mathbb E_\Lambda(f;g)|^2\leq C|\Lambda_f||\Lambda_g|e^{-M{\rm dist}(\Lambda_f,\Lambda_g)}\Big(\mathbb E_\Lambda|\nabla_\Lambda f|^2\Big)\Big(\mathbb E_\Lambda|\nabla_\Lambda g|^2\Big) \label{StrongMixing}}
for all $f,g\in H^1(\mu)$ localised in $\Lambda_f,\Lambda_g$ respectively and $C,M>0$ independent of $\Lambda,f,g$.
\end{thm}
Note that the condition on $\nabla_i\nabla_jU_\Lambda$ is trivially satisfied by the condition in Theorem \ref{muLSI} and the fact that the interaction is of finite range. Therefore, one can apply \eqref{StrongMixing} to obtain the following covariance bound 
\[\big\|\mathbb E_{\Lambda\setminus i}(\nabla_iU_{\Lambda\setminus i};\nabla_jU_{\Lambda\setminus i})\big\|_\infty^2\leq CA^4e^{2MR}(2R+1)^{4d}e^{-M|j-i|}.\] Collecting every piece together in \eqref{FUCov}, we have
\[I_2\leq\frac14CA^4e^{2MR}(2R+1)^{4d}e^{-M|j-i|}\mathbb E_\Lambda\tilde{\mathbb E}_\Lambda\Big((F-\tilde F)^2(x_i-\tilde x_i)^2\Big).\]
For simplicity let $\hat F\equiv F-\tilde F$, $\hat x_i\equiv x_i-\tilde x_i$ and $\hat{\mathbb E}_\Lambda\equiv\mathbb E_\Lambda\tilde{\mathbb E}_\Lambda$, one has the following relative entropy inequality 
\eq{\hat{\mathbb E}_\Lambda\Big(\hat F^2\hat x_i^2\Big)\leq\frac1\varepsilon\hat{\mathbb E}_\Lambda\bigg(\hat F^2\log\frac{\hat F^2}{\hat{\mathbb E}_\Lambda\hat F^2}\bigg)+\frac1\varepsilon\bigg(\log e^{\varepsilon\hat{\mathbb E}_\Lambda\hat x_i^2}\bigg)\hat{\mathbb E}_\Lambda\hat F^2 \label{relaEnt}}
for any $\varepsilon>0$. By the uniform LSI \eqref{LambdaLSI} and Lemma \ref{LSISGItensorise}, the double measure $\hat{\mathbb E}_\Lambda$ satisfies a LSI with coefficient $\tilde c$ and thus the entropy term is bounded by $\tilde c\,\hat{\mathbb E}_\Lambda(|\nabla_{x_\Lambda}\hat F|^2+|\nabla_{\tilde x_\Lambda}\hat F|^2)=2\tilde c\,\mathbb E_\Lambda|\nabla_iF|^2$, where only $\nabla_i$ remains since $F$ has no dependence in $\Lambda\setminus i$. By Lemma \ref{LSItoSGI} the measure $\hat{\mathbb E}_\Lambda$ also satisfies an SGI with coefficient $\tilde c/2$, hence 
\[\hat{\mathbb E}_\Lambda\hat F^2=2\mathbb E_\Lambda(F-\mathbb E_\Lambda F)^2\leq\tilde c\,\mathbb E_\Lambda|\nabla_\Lambda F|^2=\tilde c\,\mathbb E_\Lambda|\nabla_iF|^2.\]
For the log-term on $\hat{\mathbb E}_\Lambda\hat x_i^2$, we use the following lemma adapted from \cite[Theorem~4.5]{HeZeg}.
\begin{lemma}\label{expBound}
Let $\nu$ be a probability measure on $\mathbb R^n$ that satisfies a LSI with coefficient $\tilde c>0$. Let $g:\mathbb R^n\to\mathbb R$ be any differentiable function such that $|\nabla_{\mathbb R^n}g|^2\leq ag$ for some $a>0$, then for all $\varepsilon\in(0,1/a\tilde c)$ the following bound holds
\[\log\nu\big(e^{\varepsilon g}\big)\leq2\varepsilon\nu(g).\]
\end{lemma}
Note that $|\nabla_{x_i}\hat x_i^2|^2+|\nabla_{\tilde x_i}\hat x_i^2|^2=8\hat x_i^2$ and therefore by setting $\varepsilon<\frac1{8\tilde c}$ one can use Lemma \ref{expBound} to show that
\[\log e^{\varepsilon\hat{\mathbb E}_\Lambda\hat x_i^2}\leq2\varepsilon\hat{\mathbb E}_\Lambda\hat x_i^2=4\varepsilon\mathbb E_\Lambda(x_i-\mathbb E_\Lambda(x_i))^2\leq2\varepsilon\tilde c\,\mathbb E_\Lambda|\nabla_\Lambda x_i|^2=2\varepsilon\tilde c.\]
Taking this back into \eqref{relaEnt}, \eqref{FUCov}, \eqref{gradjEf2} and integrating with $\mu$, one gets
\eq{\mu\big|\nabla_j\big(\mathbb E_\Lambda |f|^2\big)^\frac12\big|^2\leq \mu|\nabla_jF|^2+C_Re^{-M|j-i|}\mu|\nabla_iF|^2 \label{sweepoutIterate}}
with $C_R\equiv\max\{1,\,\tilde c\,CA^4e^{2MR}(2R+1)^{4d}(\varepsilon^{-1}+\tilde c)\}$ independent of $i,j,\Lambda$.  

Recall that $F=(\mathbb E_{\Lambda\setminus i}|f|^2)^{1/2}$, hence we can define $F_2\equiv(\mathbb E_{\Lambda\setminus\{i,i_2\}}|f|^2)^{1/2}$ with $|j-i_2|>R$ and use \eqref{sweepoutIterate} again to bound $\nabla_jF$ by $\nabla_jF_2$ and $\nabla_{i_2}F_2$. With this idea, we may construct a sequence $\{i_n\}$ in $\Lambda$ such that $|j-i_n|$ is decreasing and the sequence stops at $i_N$ when $|j-i_{N+1}|\leq R$. Then \eqref{sweepoutIterate} can be applied inductively to bound each $\nabla_jF_n$ where $F_n\equiv(\mathbb E_{\Lambda\setminus\{i_1,i_2,...,i_n\}}|f|^2)^{1/2}$. To deal with the second term $\nabla_{i_n}F_n$ at each iteration, we shall use the following bound
\eq{\mu\big|\nabla_i\big(\mathbb E_{\Lambda'}|f|^2\big)^\frac12\big|^2\leq \mu|\nabla_if|^2+\tilde C_R\mu|\nabla_{\Lambda'}f|^2, \label{sweepoutIterate2}}
which holds for all $\Lambda'\subseteq\Lambda$, $i\notin\Lambda'$ and some $\tilde C_R\geq1$ depending only on $\tilde c,A,R,d$. \eqref{sweepoutIterate2} can be deduced by the similar strategy as for \eqref{sweepoutIterate} while keeping $f$ instead of $F$ in the beginning. Despite losing the exponential factor in this case, we can use \eqref{sweepoutIterate2} to terminate the iteration by directly getting the gradient bound on $f$. 

With \eqref{sweepoutIterate} and \eqref{sweepoutIterate2}, we now carry on the iteration as follows. At each step $n$, \eqref{sweepoutIterate} generates two gradient terms $\nabla_jF_n$ and $\nabla_{i_n}F_n$. For the $\nabla_jF_n$ term, we continue the iteration by applying \eqref{sweepoutIterate} again which generates $\nabla_jF_{n+1}$ and $\nabla_{i_{n+1}}F_{n+1}$. For the $\nabla_{i_n}F_n$ term, we terminate the iteration by applying \eqref{sweepoutIterate2} which generates $\nabla_{\Lambda\setminus\{i_1,i_2,...,i_{n-1}\}}f$. Once the iteration reaches the last step $N$ we simply use \eqref{sweepoutIterate2} to bound both $\nabla_jF_N$ and $\nabla_{i_N}F_N$, which ends the iteration with all gradient terms on $f$. The final output of this procedure is therefore
\[\mu\Big|\nabla_j\big(\mathbb E_\Lambda|f|^2\big)^\frac12\Big|^2\leq\mu|\nabla_jf|^2+C_R\tilde C_R\beta_N\mu|\nabla_{\Lambda\setminus\{i_1,...,i_N\}}f|^2+C_R\tilde C_R\sum_{n=1}^N\beta_n\mu|\nabla_{i_n}f|^2,\]
where $\beta_n\equiv\sum_{m=1}^ne^{-M|j-i_m|}\leq|\Lambda|e^{-M|j-i_n|}$ since $|j-i_n|$ is decreasing along the sequence. Rearranging the terms and merging the constants, one finally gets the sweeping out inequality \eqref{sweepout}
\[\mu\Big|\nabla_j\big(\mathbb E_\Lambda|f|^2\big)^\frac12\Big|^2\leq\sum_{i\in\mathbb Z^d}\alpha_{ji}\mu|\nabla_if|^2\]
with $0<\alpha_{ji}\leq e^{MR}C_R\tilde C_R|\Lambda|e^{-M|j-i|}$.

To get \eqref{sweepout}, let $F\equiv\mathbb E_{\Lambda\setminus i}f$ and we obtain that
\eq{|\nabla_j(\mathbb E_\Lambda f)|^2=|\nabla_j(\mathbb E_\Lambda F)|^2\leq2\mathbb E|\nabla_jF|^2+2|\mathbb E_\Lambda(F;\nabla_jU_\Lambda)|^2. \label{eqn:nablajEf2}}
Denote by $\mathcal U\equiv\mathbb E_{\Lambda\setminus i}(-\nabla_jU_\Lambda)$ and the covariance follows 
\[|\mathbb E_\Lambda(F;\nabla_jU_\Lambda)|^2=\frac14\big|\mathbb E_\Lambda\tilde{\mathbb E}_\Lambda\big((F-\tilde F)(\mathcal U-\tilde{\mathcal U})\big)\big|^2\leq\frac14\mathbb E_\Lambda\tilde{\mathbb E}_\Lambda\Big((F-\tilde F)^2(\mathcal U-\tilde{\mathcal U})^2\Big)\]
similarly as \eqref{FUCov}. Now the right hand side of \eqref{eqn:nablajEf2} is the same as \eqref{eqn:nablajEf}, up to scaling by a constant. Therefore, one can apply the same argument and conclude \eqref{sweepout2nd}.

\section{Remark on Analyticity}\label{Analyticity}

It was shown in \cite{SZ3} that, under certain technical conditions, a property of Dobrushin-Shlosman, complete analyticity, condition is equivalent to (uniform in volume and boundary conditions) Logarithmic Sobolev inequality. The proof relied on use of finite speed of propagation of information for Markov semigroups and hypercontractivity (which by the celebrated result of Len Gross is an equivalent condition to logarithmic Sobolev inequality). Below we describe a way one can obtain analyticity for infinite systems using the idea of Section \ref{LSI}.

Suppose we replace the map given by regular conditional expectation $\mathbb E_\Lambda$ by a map
\[
\mathbb E_{\Lambda,\varepsilon}(f)\equiv\frac{\mathbb E_\Lambda(e^{-\varepsilon V_\Lambda}f)}{\mathbb E_\Lambda(e^{-\varepsilon V_\Lambda})}
\]
with complex parameter $\varepsilon$ and an interaction $V_\Lambda$, such that $\mathbb E_\Lambda(e^{-\varepsilon V_\Lambda})\neq 0$. In particular this holds for sufficiently small $|\varepsilon|$ if
$V$ is a bounded measurable function. But, with additional conditions on the structure of $\mathbb E_{\Lambda,\varepsilon}$, it can also be achieved for unbounded interactions $V$. For example if the density of $\mathbb E_\Lambda$
is $e^{-U_\Lambda}/\int e^{-U_\Lambda}$
and $V_\Lambda-U_\Lambda$ is a bounded function.
We remark that for the map $\mathbb E_{\Lambda,\varepsilon}$ one can show the following 
relation
\[\mu\left|\nabla_j\mathbb E_{\Lambda,\varepsilon}(f)\right|\leq C_\varepsilon\mu\left| \nabla _j f\right| + \sum_{i\in \Lambda} C_\varepsilon \alpha_{ij}\mu\left| \nabla _i f\right| 
\]
for any sufficiently smooth complex function $f$ with some real constant $C_\varepsilon\to_{\varepsilon\to 0}1$.

With this property, defining a map  ${\Pi}_\varepsilon$ corresponding to the family $\{\mathbb E_{\Lambda,\varepsilon} \}$, one can show  for small $|\varepsilon|$ the uniform convergence of ${\Pi}_\varepsilon^n f$ for any localised Lipschitz function $f$. By our construction ${\Pi}_\varepsilon^n f$ is a sequence of analytic in $\varepsilon$ functions. Hence by the uniform convergence
the limiting function
\[
\mu_\varepsilon f \equiv \lim_{n\to\infty} {{\Pi}_\varepsilon^n}  f 
\]
is also an analytic  function of the complex parameter $\varepsilon$.
This suggest that the spectral gap is a real analytic function of $\varepsilon$ in a corresponding small neighbourhood of the original theory for which the first order sweeping out relations hold. 

\section{Bounds for Solutions of LSE} \label{LSEBound}
We now return to the logarithmic Schr\"odinger equation (LSE) problem in finite volume $\Lambda\subset\subset\mathbb Z^d$, given by \vspace{-.2cm}
\eq{\label{LSELambda}
\begin{split}
    &i\partial_t\phi=-\mathcal L_\Lambda\phi +\lambda\phi\log\frac{|\phi|^2}{\mu|\phi|^2}\\
    &\phi\vert_{t=0}=f
\end{split} \tag{LSE$_\Lambda$}
}
where $\mu$ is the infinite-volume Gibbs measure, constructed in Section \ref{2-Setup}, and $\mathcal L_\Lambda\equiv\Delta_\Lambda-\nabla_\Lambda U_\Lambda\cdot\nabla_\Lambda$. We shall consider initial data $f\in H^1(\mu)$ localised in some $\Lambda_f\subset\Lambda$ so that all initial particles are covered by the potential. Due to the finite-range interactions, the solution $\phi$ is localised in $\{j:{\rm dist}(j,\Lambda)\leq R\}$ for some $R>0$, therefore the problem (LSE$_\Lambda$) lives in finite dimensions and is well-defined. In this section we find various estimates uniformly in $\Lambda$. In the next section we show that these are crucial for constructing solutions to the infinite-volume LSE as $\Lambda\to\mathbb Z^d$.

Denote by $\norm{\cdot}_\mu\equiv\sqrt{\mu|\cdot|^2}$ the $L^2(\mu)$ norm, we first show that the mass and energy of (LSE$_\Lambda$) are conserved over time.      
\begin{thm}\label{Conservation}
For any $\Lambda\subset\subset\mathbb Z^d$, let $\phi(\cdot,t)$ be the solution of {\rm(LSE$_\Lambda$)} with initial data $f\in H^1(\mu)$. Then for all time $t\geq0$ the following conservation of mass property holds
\[\|\phi(\cdot,t)\|_\mu^2=\|f\|_\mu^2. \]
If $U_\Lambda$ is twice differentiable, the following conservation of energy property holds
\eq{\|\nabla_\Lambda\phi(\cdot,t)\|_\mu^2 +\lambda{\rm Ent}_\mu|\phi(\cdot,t)|^2
= \|\nabla_\Lambda f\|_\mu^2 +\lambda{\rm Ent}_\mu|f|^2 \label{Econv}}
where ${\rm Ent}_\mu F\equiv F\log(F/\mu F)$ denotes the $\mu$-entropy.
\end{thm}
\begin{proof}
The mass conservation easily follows from
\[\partial_t\mu|\phi|^2=2{\rm Re}(\mu(\phi^*\partial_t\phi))=-2{\rm Im}\big(\mu(\phi^*\mathcal L_\Lambda\phi)\big)=0\]
using (LSE$_\Lambda$) and the reversibility of $\mu$ for $\mathcal L_\Lambda$, due to Lemma \ref{Lem 2.3}. To obtain the energy conservation, we will similarly show that
\eq{\partial_t\big(\mu|\nabla_\Lambda\phi|^2+\lambda{\rm Ent}_\mu|\phi|^2\big)=0. \label{EtimeDerivative}}
In what follows we consider regularised solution $e^{-\varepsilon \mathcal L_\Lambda}\phi$ of (LSE$_\Lambda$) and pass to the limit with $\varepsilon$ to $0$ at the end. Keeping this in mind, below we provide the stream of essential steps omitting explicit reference to $\varepsilon$. First we note that for any $j\in\mathbb Z^d$ the time evolution of $|\nabla_j\phi|^2$ follows
\eq{\partial_t|\nabla_j\phi|^2=2{\rm Im}\bigg(-\nabla_j(\mathcal L_\Lambda\phi)\cdot\nabla_j\phi^*+\lambda\nabla_j\bigg(\phi\log\frac{|\phi|^2}{\mu|\phi|^2}\bigg)\cdot\nabla_j\phi^*\bigg). \label{partialtgradjphi}}
Rewrite the first product as
\[\nabla_j(\mathcal L_\Lambda\phi)\cdot\nabla_j\phi^*=[\nabla_j,\mathcal L_\Lambda]\phi\cdot\nabla_j\phi^*+(\nabla_j\phi^*)\mathcal L_\Lambda(\nabla_j\phi) \]
where the commutator follows $[\nabla_j,\mathcal L_\Lambda]=-\sum_{k\in\Lambda}\nabla_{jk}U_\Lambda\cdot\nabla_k$ with $\nabla_{jk}\equiv\nabla_j\nabla_k$. Inserting this into \eqref{partialtgradjphi} and integrating with $\mu$, one gets
\begin{align}
    \partial_t\mu|\nabla_j\phi|^2=\ &{\rm Im}\bigg(2\mu\sum_{k\in\Lambda}\nabla_{jk}U_\Lambda\cdot\nabla_k\phi\cdot\nabla_j\phi^*\bigg)-{\rm Im}\bigg(2\mu\Big((\nabla_j\phi^*)\mathcal L_\Lambda(\nabla_j\phi)\Big)\bigg) \label{gradphi1}\\
    &+{\rm Im}\Bigg(2\lambda\mu\bigg(\nabla_j\bigg(\phi\log\frac{|\phi|^2}{\mu|\phi|^2}\bigg)\cdot\nabla_j\phi^*\bigg)\Bigg). \label{gradphi2}
\end{align}
By summing over $j\in\Lambda$, the first term in \eqref{gradphi1} vanishes due to symmetry and the second term also vanishes by the reversibility of $\mu$. Therefore we only have the last term being summed over, which gives
\eq{\partial_t\mu|\nabla_\Lambda\phi|^2={\rm Im}\Big(2\lambda\mu\big(\nabla_\Lambda(\phi\log|\phi|^2)\cdot\nabla_\Lambda\phi^*\big)\Big). \label{gradjphilog}}
For the entropy term, using mollification argument with the invariance of the $L^2$ property we have
\[\frac\partial{\partial t}{\rm Ent}_\mu|\phi|^2=
-{\rm Im}\Big(2\mu\big(\phi^*\log|\phi|^2\mathcal L_\Lambda\phi\big)\Big)-{\rm Im}\Big(2\mu\Big(\phi^*\mathcal L_\Lambda\phi+\lambda(1+\log|\phi|^2){\rm Ent}_\mu|\phi|^2\Big)\Big).\]
The second term vanishes again due to the reversibility of $\mu$ and therefore we get
\eq{\frac\partial{\partial t}{\rm Ent}_\mu|\phi|^2=-{\rm Im}\Big(2\mu\big(\phi^*\log|\phi|^2\mathcal L_\Lambda\phi\big)\Big)={\rm Im}\Big(2\mu\big(\nabla_\Lambda(\phi^*\log|\phi|^2)\cdot\nabla_\Lambda\phi\big)\Big) \label{entboundstep}}
using integration by parts. By noticing that the measure in \eqref{entboundstep} is the complex conjugate of the measure in \eqref{gradjphilog}, one easily verifies \eqref{EtimeDerivative} and hence the energy conversation property.
\end{proof}

\begin{remark}
The mass conservation property holds true with any measure $\mu_\Lambda$ and the operator $\mathcal L_\Lambda$ corresponding to the Dirichlet form $\mu_\Lambda(f \mathcal L_\Lambda g) = -\mu_\Lambda(\nabla_\Lambda f \cdot \nabla_\Lambda g).$
\end{remark} 

With similar steps one can construct time-dependent bounds for the gradient and the entropy. In particular, we show that these bounds can be uniform in time if $\mu$ satisfies a log-Sobolev inequality (e.g. it satisfies Theorem \ref{LSImu}).

\begin{thm}
\label{GradEntbound}
For any $\Lambda\subset\subset\mathbb Z^d$ let $\phi$ be the solution of $({\rm LSE}_\Lambda)$ with localised initial data $f\in H^1(\mu)$. If $U_\Lambda$ is twice differentiable, then for all $t\geq0$ one has the gradient bound 
\eq{\|\nabla_\Lambda\phi(\cdot,t)\|_\mu^2\leq e^{2|\lambda|t}\|\nabla f\|_\mu^2
\label{gradbound}}
and the entropy bound
\vspace{-.3cm}
\begin{equation}\label{Entbound}
{\rm Ent}_\mu|\phi(\cdot,t)|^2\leq{\rm Ent}_\mu|f|^2+\frac{1}{|\lambda|} (e^{2|\lambda| t }-1 )\|\nabla f\|_\mu^2.
\end{equation}
If $\lambda\,{>}\,0$ and $\mu$ satisfies a log-Sobolev inequality with coefficient $c\,{>}\,0$, one has the time-uniform bounds
\eq{\|\nabla_\Lambda\phi(\cdot,t)\|_\mu^2\leq(1+\lambda c)\|\nabla f\|_\mu^2\quad\text{and}\quad{\rm Ent}_\mu|\phi(\cdot,t)|^2\leq(1/\lambda+c)\|\nabla f\|_\mu^2.
\label{uniformBound}}
\end{thm}

\begin{proof}
Following the proof of Theorem \eqref{Conservation}, the estimate \eqref{gradjphilog} leads to
\[\partial_t\mu|\nabla_\Lambda\phi|^2\leq2|\lambda|\sum_{j\in\Lambda}\mu\Big|\,{\rm Im}\big(\phi\nabla_j\phi^*\big)\big(\nabla_j\log|\phi|^2\big)\Big|\leq2|\lambda|\sum_{j\in\Lambda}\mu|\nabla_j\phi|^2.\]
Using Gr\"onwall's lemma one easily obtains the gradient bound \eqref{gradbound}. For the entropy bound, a follow-up calculation of \eqref{entboundstep} shows that
\[\frac\partial{\partial t}{\rm Ent}_\mu|\phi|^2=-2\sum_{j\in\Lambda}{\rm Im}\bigg(\mu\bigg(\frac{\phi^*}{\phi}(\nabla_j\phi)^2\bigg)\bigg)\leq2\mu|\nabla_\Lambda\phi|^2\leq2e^{2|\lambda|t}\mu|\nabla f|^2,\]
where the last inequality is by \eqref{gradbound} that we just verified. Integrating with respect to $t$ and one obtains the entropy bound \eqref{Entbound}. The uniform bounds \eqref{uniformBound} easily follow from the log-Sobolev inequality of $\mu$ and the energy conservation property \eqref{Econv}.
\end{proof}

With the gradient bound of $\nabla_\Lambda\phi$ inside $\Lambda$, we now proceed to estimate $\nabla_j\phi$ for all $j\in\mathbb Z^d$. Due to the finite-range interactions, one can show that the norm of $\nabla_j\phi$ decays exponentially as $j$ moves away from the initial region $\Lambda_f$. This property corresponds to the finite speed propagation of information, which is presented in the following theorem.

\begin{thm} \label{finiteprop}
For any $\Lambda\subset\subset\mathbb Z^d$ let $\phi$ be the solution of {\rm(LSE$_\Lambda$)} with initial data $f\in H_1(\mu)$ localised in some $\Lambda_f\subset\subset\mathbb Z^d$. Suppose the potential $U_\Lambda$ satisfies $\|\nabla_j\nabla_kU_\Lambda\|_\infty\leq A$ for all $\Lambda\subset\subset\mathbb Z^d$, $j,k\in\mathbb Z^d$ with $j\neq k$ and some $A>0$ independent of $i,j,\Lambda$. Then the following estimate holds
\[\|\nabla_j\phi(\cdot,t)\|_\mu^2\leq e^{-N_j}\|\nabla f\|_\mu^2 \]
for any $j\in\mathbb Z^d$ and all time $t\leq \varepsilon N_j$ where $\varepsilon^{-1}\equiv9(2|\lambda|+A)(2R+1)^{2d}$ is the propagation speed and $N_j$ is the distance $N_j\equiv\left\lfloor{\rm dist}(j,\Lambda_f)R^{-1}\right\rfloor$.
\end{thm}

\begin{proof} 
\noindent Using the calculation in \eqref{gradphi1}, \eqref{gradphi2} and \eqref{gradjphilog}, we have for any $j\in\mathbb Z^d$ the estimate
\eq{\partial_t\big(\mu|\nabla_j\phi|^2\big)\leq\,2|\lambda|\,\mu|\nabla_j\phi|^2+\sum_{k\in\Lambda\setminus\{j\}}\|\nabla_{jk}U_\Lambda\|_\infty\big(\mu|\nabla_k\phi|^2+\mu|\nabla_j\phi|^2\big). \label{gradjIneq1}}
Assuming the uniform bound on $|\nabla_{jk}U_\Lambda|$ holds, one gets the following differential inequality
\eq{\partial_t\big(\mu|\nabla_j\phi|^2\big)\leq\sum_{k\in\mathbb Z^d}A_{jk}\,\mu|\nabla_k\phi|^2 \label{sumAjkbound}}
where $A_{jj}=2|\lambda|+A(2R+1)^d$, $A_{jk}=A$ if $0<|j-k|\leq R$ and $A_{jk}=0$ otherwise.
We will use this inequality to get an improved estimate on $\|\nabla_j\phi\|_\mu$ when the initial data $f$ is a local function. First we remark that as a consequence of \eqref{gradbound} and \eqref{gradjIneq1} it follows
\[\partial_t\mu|\nabla_j\phi|^2\leq A_{jj}\mu|\nabla_j\phi|^2+A\mu|\nabla_\Lambda\phi|^2\leq A_{jj}\mu|\nabla_j\phi|^2+Ae^{2|\lambda|t}\mu|\nabla f|^2.\]
By Theorem A in \cite{YJY}, the above inequality implies $\mu|\nabla_j\phi(\cdot,t)|^2\leq ae^{\varepsilon t}\mu|\nabla f|^2$ with $a\equiv 1+\frac A{2|\lambda|}$ and $\varepsilon\equiv A_{jj}+2|\lambda|$ independent of $j$. Therefore, the following estimate
\eq{\bar\phi_t \equiv\sup_{k\in\mathbb Z^d}\sup_{s\in[0,t]}\mu|\nabla_k\phi(\cdot,s)|^2\leq ae^{\varepsilon t}\mu|\nabla f|^2 \label{barphi}}
holds for any finite time $t\geq0$. Let $f\in H_1(\mu)$ be localised in some finite volume $\Lambda_f\subset\subset\mathbb Z^d$. Taking into account that the interaction is of finite range $R$ we will show that within a fixed time period $\|\nabla_j\phi\|_\mu$ decays exponentially with ${\rm dist}(j,\Lambda_f)$. To this end, let $N_j\equiv\left\lfloor{\rm dist}(j,\Lambda_f)R^{-1}\right\rfloor$ and consider the integral inequality
\eq{\mu|\nabla_j\phi(\cdot,t)|^2\leq\mu|\nabla_jf|^2+\sum_{k\in\mathbb Z^d}A_{jk}\int_0^tdt_1\,\mu|\nabla_k\phi(\cdot,t_1)|^2 \label{integralAjkbound} }
deduced from \eqref{sumAjkbound}. Note that $\nabla_jf$ is zero if $j\notin\Lambda_f$, thus we can iterate \eqref{integralAjkbound} for $M>N_j$ times until we get a nonzero derivative of $f$. 
This yields
\begin{align*}
    \mu|\nabla_j\phi(\cdot,t)|^2&\leq\sum_{n=N_j}^M\frac{t^n}{n!}\sum_{k_1,...,k_n\in\mathbb Z^d}(A_{jk_1}...\,A_{k_{n-1}k_n} )\mu|\nabla_{k_n}f|^2 \nonumber\\
    &+\sum_{k_1,...,k_{M+1}\in\mathbb Z^d}(A_{jk_1}...\,A_{k_Mk_{M+1}} )\int_0^t\cdots\int_0^{t_M}dt_{M+1}\,\mu\left|\nabla_{k_{M+1}}\phi(\cdot,t_{M+1})\right|^2.
\end{align*}
Denoting by $A_*=A_{jj}(2R+1)^d$, the above sum is then bounded by  
\[\left(\sum_{n=N_j}^M\frac{(A_*t)^n}{n!}\mu|\nabla  f|^2 \right)+\frac{(A_*t)^M}{M!}\bar\phi_t \]
with $\bar\phi_t$ defined in \eqref{barphi}. Since $\bar\phi_t<\infty$ for all finite time $t$, sending $M\to\infty$ leads to 
\[\mu|\nabla_j\phi(\cdot,t)|^2\leq\sum_{n=N_j}^\infty\frac{(A_*t)^n}{n!}
\mu|\nabla  f|^2 \leq\frac{(A_*t)^{N_j}}{N_j!}e^{A_*t}\mu|\nabla  f|^2,
\]
where by Stirling's formula $n!\geq n^ne^{-n}$ one has
\[\frac{(A_*t)^{N_j}}{N_j!}e^{A_*t}\leq\exp\bigg(N_j\bigg(\log(A_*)+1-\log\frac{N_j}t+A_*\frac t{N_j}\bigg)\bigg).\]
For all $t\leq\frac{N_j}{9A_*}$ it easily follows $\log(A_*)+1-\log(N_j/t)+A_*(t/N_j)\leq-1$, which concludes the proof of Theorem \thmref{finiteprop}.
\end{proof}

We note that where $U_\Lambda$ is given by the bilinear or perturbed bilinear potentials, detailed in Section \ref{LSI}, we have the above estimates in Theorem \ref{Conservation}, \ref{GradEntbound}, \ref{finiteprop}. Recall that $U_\Lambda$ is the potential for the ground state representation of the original solution $\psi\equiv\phi e^{-\frac12U_\Lambda}$. For consistency, in the following remark we relate $U_\Lambda$ back to the original potential $V_\Lambda$ by the differential equation \eqref{UV}, namely $V_\Lambda=-\frac12\Delta_\Lambda U_\Lambda+\frac14|\nabla_\Lambda U_\Lambda|^2+\lambda U_\Lambda$.

\begin{remark}
Consider the bilinear example of $U_\Lambda$ with conditions given in Example \ref{Exmp:bilinear}. Assuming $C_{jk}=C_{kj}$ for simplicity we have
\eq{V_\Lambda(x)=D_\Lambda+\sum_{j,k\in\Lambda_R}D_{jk}x_jx_k, \label{Vbilinear}}
where $\Lambda_R\equiv\{j:{\rm dist}(j,\Lambda)\leq R\}$,\ \,$D_\Lambda\equiv-\text{\raisebox{0.3pt}{\scalebox{0.9}{$\sum_{j\in\Lambda}$}}}C_{jj}$, and\ \,$D_{jk}\equiv\lambda C_{jk}{\bf1}_{(j\in\Lambda)\lor(k\in\Lambda)}+\text{\raisebox{0.3pt}{\scalebox{0.9}{$\sum_{l\in\Lambda}$}}}C_{jl}C_{kl}$. With the conditions on $C_{jk}$, it follows that $D_{jk}$ is uniformly bounded and of finite range $2R$. In case $\lambda$ is positive, the quadratic coefficient $D_{jj}$ is uniformly bounded below by some positive constant. If $\lambda$ is negative and large in absolute value (e.g. $\lambda\,{<}\,{-}\sup_jC_{jj}$ when $C_{ij}$ is diagonal), the quadratic terms become negative and delicate issues may arise. For the latter case, we note that LSE with negative bilinear $V$ part were recently considered in \cite{CS} as an interesting mathematics situation where non-unique positive stationary solutions could exist, with each one generating a continuous family of solitary waves. We remark that large negative values of $\lambda$ correspond in our setup to the loss of log-Sobolev inequality under which one has uniqueness and strict positivity of the ground state.

For the perturbed bilinear case in Example \ref{Exmp:bilinearPerturbed}, denote by $Y_\Lambda\equiv\text{\raisebox{0.3pt}{\scalebox{0.9}{$\sum_{X\subset\subset\mathbb Z^d:X\cap\Lambda\neq\varnothing}$}}}W_X$ and assume $C_{jk}$ is symmetric for simplicity, then in this case $V_\Lambda$ takes the quadratic form
\[V_\Lambda(x)=D_\Lambda+\sum_{k\in\Lambda_R}D_kx_k+\sum_{j,k\in\Lambda_R}D_{jk}x_jx_k,\]
where $\Lambda_R$ and $D_{jk}$ are defined as in \eqref{Vbilinear} with the coefficient $D_k\equiv\varepsilon\text{\raisebox{0.3pt}{\scalebox{0.9}{$\sum_{j\in\Lambda}$}}}C_{jk}(\nabla_jY_j)$ and the bounded perturbation 
\[D_\Lambda=\lambda\varepsilon Y_\Lambda+\sum_{j\in\Lambda}\Big(-C_{jj}+\frac{\varepsilon^2}4|\nabla_jY_j|^2-\frac\varepsilon2\Delta_jY_j\Big).\]
It is easy to check that $V_\Lambda$ has uniformly bounded coefficients $D_k,D_{jk}$ and is of finite range $2R$.
\end{remark}

Finally, for any $\Lambda\subset\subset\mathbb Z^d$ consider the setting with removed boundary conditions, $(U_\Lambda^\circ,\mathbb E_\Lambda^\circ)$ defined in Example \ref{mucirc}, where $U_\Lambda^\circ\equiv\text{\raisebox{0.2pt}{\scalebox{0.9}{$\sum_{X\subseteq\Lambda}$}}}J_X$ is the potential strictly inside $\Lambda$ and $\mathbb E_\Lambda^\circ$ is the local Gibbs measure corresponding to $U_\Lambda^\circ$. The LSE in $\Lambda$ with no boundary conditions is given by
\eq{\label{LSEcirc}
\begin{split}
    &i\partial_t\phi_\Lambda=-\mathcal L_\Lambda^\circ\phi_\Lambda+\lambda\phi_\Lambda\log\frac{|\phi_\Lambda|^2}{\mathbb E_\Lambda^\circ|\phi_\Lambda|^2}\\
    &\phi_\Lambda\vert_{t=0}=f
\end{split} \tag{LSE$_\Lambda^\circ$}
}
with $\mathcal L_\Lambda^\circ\equiv\Delta_\Lambda-\nabla_\Lambda U_\Lambda^\circ\cdot\nabla_\Lambda$ and initial condition $f$ localised in $\Lambda_f\subset\Lambda$. For this setup, we have the following results.

\begin{coro} \label{coro:circEstimates}
If the $\mu$-estimates in Theorem \ref{Conservation}, \ref{GradEntbound}, \ref{finiteprop} hold for \eqref{LSELambda}, they remain true for \eqref{LSEcirc} with $\mu$ replaced by $\mathbb E_\Lambda^\circ$. That is, for the same conditions on $U_\Lambda^\circ$, one has the conservation property
\[\mathbb E_\Lambda^\circ|\phi_\Lambda|^2=\mathbb E_\Lambda^\circ|f|^2\quad{\rm and}\quad\mathbb E_\Lambda^\circ|\nabla\phi_\Lambda|^2 +\lambda{\rm Ent}_{\mathbb E_\Lambda^\circ}|\phi_\Lambda|^2
= \mathbb E_\Lambda^\circ|\nabla f|^2 +\lambda{\rm Ent}_{\mathbb E_\Lambda^\circ}|f|^2,\]
the gradient and entropy bounds 
\[\mathbb E_\Lambda^\circ|\nabla\phi_\Lambda|^2\leq e^{2|\lambda|t}\mathbb E_\Lambda^\circ|\nabla f|_\mu^2\quad{\rm and}\quad{\rm Ent}_{\mathbb E_\Lambda^\circ}|\phi_\Lambda|^2\leq{\rm Ent}_{\mathbb E_\Lambda^\circ}|f|^2+|\lambda|^{-1}(e^{2|\lambda| t }-1 )\mathbb E_\Lambda^\circ|\nabla f|^2,\]
and the finite-speed propagation property 
\eq{\mathbb E_\Lambda^\circ|\nabla_j\phi_\Lambda|^2\leq e^{-N_j}\mathbb E_\Lambda^\circ|\nabla f|^2 \label{eqn:circfiniteprop}}
for all $j\in\mathbb Z^d$ and time $t\leq \varepsilon N_j$, where $\varepsilon$ and $N_j$ are defined as in Theorem \ref{finiteprop}.
\end{coro}
Note that in the above theorem we have uniform estimates for the total gradient $\nabla\phi_\Lambda$ instead of $\nabla_\Lambda\phi_\Lambda$ since the solution $\phi_\Lambda$ of \eqref{LSEcirc} is localised entirely in $\Lambda$ due to non-existence of boundary conditions. In the next section, we will use the finite-volume problem \eqref{LSEcirc} with the estimates in Corollary \ref{coro:circEstimates} to construct solutions of the infinite-volume LSE problem.

\section{Existence of Weak Solutions in Infinite Volume}\label{Convergence}
In this section we establish the existence of weak solutions to the infinite-volume LSE by finite-dimensional approximations. Namely, we will construct a sequence $\{\Lambda_n\}$, increasing in volume, such that the sequence of solutions $\phi_{\Lambda_n}$ of (LSE$_{\Lambda_n}^\circ$) is compact in $H^1(\mu)$. Then we choose a suitable convergent subsequence whose limit defines a weak infinite-volume solution.

We consider \eqref{LSEcirc} since it lives entirely inside $\mathbb R^{\Lambda}$, which is technically simpler and for which we already have existence and uniqueness of the solutions. When studying uniqueness of the infinite-volume solution, it is therefore convenient to compare the solution obtained from \eqref{LSEcirc} with the one from other setups such as \eqref{LSELambda}.

For simplicity, we shall hereafter work on the space $\Omega\equiv\mathbb R^{\mathbb Z}$ and assume the multi-particle interaction is bounded, i.e. the potentials $J_X$ have bounded derivatives of order $n=0,1,2$ uniformly in $X$ for all $|X|\geq2$. For the higher dimensional case see Remark \ref{rmk:1dLattice}. Moreover, we assume that the measures $\mathbb E_\Lambda$ and $\mathbb E_\Lambda^\circ$ associated to $\{J_X\}$ satisfy a log-Sobolev inequality with coefficient uniform in $\Lambda$. With these conditions, we are ready to introduce the existence theorem.
\begin{thm} \label{existence}
Let $\Omega=\mathbb R^{\mathbb Z}$ and suppose $\{J_X\}$ satisfies the above conditions. For any $\lambda\in\mathbb R$ and $f\in H^1(\mu)$, if the Cauchy problem {\rm\eqref{LSEcirc}} with initial data $f$ and coupling constant $\lambda$ admits at least one solution for all $\Lambda\subset\subset\mathbb Z^d$, then there exists a weak solution to the infinite-volume LSE problem with initial data $f$ and coupling constant $\lambda$, i.e. there exists $\phi\in H^1(\mu)$ such that
\eqN{
\begin{split}
    &\mu\left(g i\partial_t \phi\right) = \mu\bigg(\nabla g\cdot\nabla\phi+g\lambda \phi\log\frac{|\phi|^2}{\mu|\phi|^2}\bigg)\\
    &\phi(t=0)=f
\end{split}
}
for all smooth and compactly supported local  functions $g:\Omega\to\mathbb C$.
\end{thm}

\begin{proof}[Proof of Theorem \ref{existence}]

Let $f\in H^1(\mu)$ be localised in $\Lambda_f\subset[-L,L]$ for some $L\geq R$, where $R$ is the interaction range. Consider the sequence $\Lambda_n\equiv[-8nL,8nL]\cap\mathbb Z$ and suppose \eqref{LSEcirc} with initial condition $f$ admits at least one solution  $\phi_n\equiv\phi_{\Lambda_n}$ for all $n\in\mathbb N$. We will prove the following convergence of the norms 
\eq{
\lim_{n\to\infty} \mu|\phi_n|^2=\mu|f|^2. \label{normConvergence}}
Denote by $\rho_n\equiv d\mathbb E_{\Lambda_n}/d\mathbb E_{\Lambda_n}^\circ$ the density ratio, we can rewrite $\mu|\phi_n|^2=\mu\mathbb E_\Lambda|\phi_n|^2=\mu\mathbb E_\Lambda^\circ(|\phi_n|^2\rho_n)$. With the notation $\nu_n\equiv\mu\mathbb E_{\Lambda_n}^\circ$ and a notation $\nu(g;h)\equiv\nu(gh)-\nu(g)\nu(h)$ for the $\nu$-covariance, we then have 
\eq{\mu|\phi_n|^2=\nu_n(|\phi_n|^2;\rho_n)+\big(\nu_n|\phi_n|^2\big)\big(\nu_n(\rho_n)\big)=\nu_n(|\phi_n|^2;\rho_n)+\nu_n|f|^2 \label{eqn:phicov}}
since $\mathbb E_{\Lambda_n}^\circ(\rho_n)=1$ and $\mathbb E_{\Lambda_n}^\circ|\phi_n|^2=\mathbb E_{\Lambda_n}^\circ|f|^2$ by the mass conservation property in Corollary \ref{coro:circEstimates}. Using \eqref{eqn:phicov} for the norms of $\phi_n$ and $\phi_{n+1}$, one has the following estimate
\eq{\Big|\mu |\phi_{n+1}|^2 - \mu |\phi_{n}|^2\Big|\leq
\Big|\nu_n\big(|\phi_{n}|^2;\rho_n\big)\Big|+\Big|\nu_{n+1}\big(|\phi_{n+1}|^2;\rho_{n+1}\big)\Big|+\Big|\nu_{n+1}|f|^2-\nu_n|f|^2\Big|. 
\label{muphidifference}}

To bound the first covariance, let $\Gamma=([-12nL,-4nL]\cup[4nL,12nL])\cap\mathbb Z$ and we rewrite the covariance by
\eq{\nu_n(|\phi_{n}|^2;\rho_n)=\nu_n\big(|\phi_{n}|^2-\mathbb E_\Gamma|\phi_n|^2;\rho_n\big)+\nu_n(\mathbb E_\Gamma|\phi_n|^2;\rho_n). \label{eqn:nucov}}
To proceed, we first note that if the lattice is one-dimensional and $J_X$ is uniformly bounded for $|X|\geq2$ then $\rho_n$ is bounded uniformly in $n$ by $1/B_1\leq\rho_n\leq B_1$, where
\eq{B_1=\exp\bigg(4\sup_{k\in\mathbb Z}\bigg(\text{\scalebox{0.8}{\raisebox{2pt}{$\sum_{\substack{X\ni\,k,\,|X|\geq2}}$}}\scalebox{0.9}{$\|J_X\|_\infty$}}\bigg)\bigg) \label{densityBound}}
is finite since the interaction is of finite range. Hence, for any function $g\geq0$ one has $\mu(g)\leq B_1\nu_n(g)$ and $\nu_n(g)\leq B_1\mu(g)$. With this density bound, the first covariance in \eqref{eqn:nucov} can be estimated by
\eq{\big|\nu_n\big(|\phi_n|^2-\mathbb E_\Gamma|\phi_n|^2;\rho_n\big)\big|\leq 2B_1\mu\big||\phi_n|^2-\mathbb E_\Gamma|\phi_n|^2\big|\leq2B_1\mu\mathbb E_\Gamma\tilde{\mathbb E}_\Gamma\big||\phi_n|^2-|\tilde\phi_n|^2\big|. \label{eqn:EgammaDouble}}
Applying the same argument as in \eqref{FUCov} on the double measure $\mathbb E_\Gamma\tilde{\mathbb E}_\Gamma$, it follows
\[\eqref{eqn:EgammaDouble}\leq4B_1\mu\Big(\big(\mathbb E_\Gamma|\phi_n|^2\big)^{1/2}\big(\mathbb E_\Gamma|\phi_n-\mathbb E_\Gamma\phi_n|^2\big)^{1/2}\Big)\leq4(B_1)^2\big(\mathbb E_\Gamma^\circ|f|^2\big)^{1/2}\mu\big(\mathbb E_\Gamma|\nabla_\Gamma\phi_n|^2\big)^{1/2},\]
where we used the mass conservation for $\mathbb E_\Gamma^\circ$ and the spectral gap inequality for $\mathbb E_\Gamma$. Using the finite-speed propagation property \eqref{eqn:circfiniteprop} for $\mathbb E_\Gamma^\circ|\nabla_\Gamma\phi_n|^2$, one has the estimate 
\[\eqref{eqn:EgammaDouble}\leq D_1|\Lambda_n|e^{-\lfloor{\rm dist}(\Gamma,\Lambda_f)/R\rfloor}(\mu|f|^2)^{1/2}(\mu|\nabla f|^2)^{1/2}\]
with $D_1\equiv4(B_1)^5$. By the construction of $\Gamma$ and $\Lambda_n$, we have ${\rm dist}(\Gamma,\Lambda_f)\geq\frac18{\rm diam}(\Lambda_n)$ and therefore there exists $\varepsilon_1>0$ independent of $n$ such that
\eq{\big|\nu_n\big(|\phi_n|^2-\mathbb E_\Gamma|\phi_n|^2;\rho_n\big)\big|\leq D_1|\Lambda_n|e^{-\varepsilon_1{\rm diam}(\Lambda_n)}\big(\mu|f|^2+\mu|\nabla f|^2\big). \label{eqn:phincov1}}
Here we note that due to the finite-speed propagation property \eqref{eqn:circfiniteprop}, the estimate \eqref{eqn:phincov1} only holds in the time range $t\leq C{\rm diam}(\Lambda_n)$ for some $C>0$ independent of $n$. As $n\to\infty$ in the sequence, this time range will eventually extend to $[0,\infty)$ and the above estimate becomes global in time. 

To bound the second covariance in \eqref{eqn:nucov}, we use the mixing condition \eqref{StrongMixing} for the measure $\mathbb E_{\Lambda_n}^\circ$, which holds true since $\mathbb E_{\Lambda_n}^\circ$ satisfies a uniform LSI and $\nabla_{ij}U_\Lambda^\circ$ is bounded due to the bounded multi-particle interaction. Therefore, there exist $C_{\rm mix},M>0$ independent of $n$ such that
\[\big|\mathbb E_{\Lambda_n}^\circ(\mathbb E_\Gamma|\phi_n|^2;\rho_n)\big|\leq C_{\rm mix}|\Lambda_\Gamma||\Lambda_{\rho_n}|e^{-M{\rm dist}(\Gamma^{\mathsf c},\Lambda_{\rho_n})}\Big(\mathbb E_{\Lambda_n}^\circ\big|\nabla_{\Lambda_n}\mathbb E_\Gamma|\phi_n|^2\big|^2\Big)^\frac12\Big(\mathbb E_{\Lambda_n}^\circ|\nabla_{\Lambda_n}\rho_n|^2\Big)^\frac12\]
where $\Lambda_\Gamma,\,\Lambda_{\rho_n}$ are the sets where $\mathbb E_\Gamma|\phi_n|^2$ and $\rho_n$ are localised, respectively. First, we note that both $|\Lambda_\Gamma|$ and $|\Lambda_{\rho_n}|$ are smaller than $|\Lambda_n|$. Moreover, since ${\rm dist}(\Gamma^{\mathsf c},\Lambda_{\rho_n})\geq\frac18{\rm diam}(\Lambda_n)$ there exists $\varepsilon_2>0$ independent of $n$ such that the exponential term is bounded by $e^{-\varepsilon_2{\rm diam}(\Lambda_n)}$. To estimate the $\nabla_{\Lambda_n}\rho_n$ term, we use the following bound
\[\|\nabla_{\Lambda_n}\rho_n\|_\infty\leq 2B_1\sup_{k\in\mathbb Z}\bigg(\text{\scalebox{0.8}{\raisebox{2pt}{$\sum_{\substack{X\ni\,k,\,|X|\geq2}}$}}\scalebox{0.9}{$\|\nabla J_X\|_\infty$}}\bigg)\equiv B_2\]
where $B_2$ is finite since $J_X$ is of finite range and has uniformly bounded derivatives for $|X|\geq2$. For the gradient $\nabla_{\Lambda_n}\mathbb E_\Gamma|\phi_n|^2$, direct calculation gives  
\[\nabla_j\mathbb E_\Gamma|\phi_n|^2=\mathbb E_\Gamma\big(\nabla_j|\phi_n|^2)-\mathbb E_\Gamma\big(|\phi_n|^2;\nabla_jU_\Gamma)\]
for any $j\notin\Gamma$. The first term is simply bounded by $\mathbb E_\Gamma|\phi_n|^2+\mathbb E_\Gamma|\nabla_j\phi_n|^2$ due to Young's inequality. For the covariance, an application of the relative entropy inequality (similar as \eqref{relaEnt}) and Lemma \ref{expBound} gives
\[\big|\mathbb E_\Gamma\big(|\phi_n|^2;\nabla_jU_\Gamma)\big|\leq C_{RE}|\Lambda_n|\big(\mathbb E_\Gamma|\phi_n|^2+\mathbb E_\Gamma|\nabla_j\phi_n|^2\big)\]
for some $C_{RE}>0$ independent of $n$. Combining all the estimates, we have
\[\big|\mu\mathbb E_{\Lambda_n}^\circ(\mathbb E_\Gamma|\phi_n|^2;\rho_n)\big|\leq B_2C_{\rm mix}(1+C_{RE})|\Lambda_n|^4e^{-\varepsilon_2{\rm diam}(\Lambda_n)}\big(\mu|\phi_n|^2+\mu|\nabla\phi_n|^2\big).\]
Using the mass conservation for $\mathbb E_{\Lambda_n}^\circ|\phi_n|^2$ and the gradient bound for $\mathbb E_{\Lambda_n}^\circ|\nabla\phi_n|^2$, it follows
\eq{\big|\mu\mathbb E_{\Lambda_n}^\circ(\mathbb E_\Gamma|\phi_n|^2;\rho_n)\big|\leq D_2|\Lambda_n|^4e^{-\varepsilon_2{\rm diam}(\Lambda_n)}\big(\mu|f|^2+e^{2|\lambda|t}\mu|\nabla f|^2\big) \label{eqn:phincov2}}
with $D_2\equiv B_1^2B_2C_{\rm mix}(1+C_{RE})$. Taking \eqref{eqn:phincov1} and \eqref{eqn:phincov2} into \eqref{eqn:nucov}, we arrive at
\eq{\big|\nu_n(|\phi_{n}|^2;\rho_n)\big|\leq D_{12}|\Lambda_n|^4e^{-\varepsilon_{12}{\rm diam}(\Lambda_n)}\big(\mu|f|^2+e^{2|\lambda|t}\mu|\nabla f|^2\big) \label{eqn:nucovbound}}
with $D_{12}\equiv\max\{D_1,D_2\}$ and $\varepsilon_{12}\equiv\min\{\varepsilon_1,\varepsilon_2\}$. This concludes the estimate of the first term in \eqref{muphidifference}. The second term in \eqref{muphidifference} follows similarly and therefore we will finish the proof of \eqref{normConvergence} by estimating the third term in \eqref{muphidifference}.

To this end, for any $\theta\in[0,1]$ we define the interpolating potential
\[U_\theta^\circ\equiv U_{\Lambda_n}^\circ+\text{\scalebox{0.8}{\raisebox{2pt}{$\sum_{k\in\Lambda_{n+1}\setminus\Lambda_n}$}}\scalebox{0.9}{$J_{\{k\}}$}}+\text{\scalebox{0.8}{\raisebox{2pt}{$\sum_{\substack{X\subseteq\Lambda_{n+1},\,|X|\geq2\\X\cap(\Lambda_{n+1}\setminus\Lambda_n)\neq\varnothing}}$}}\scalebox{0.9}{$\theta J_X$}}\]
and the corresponding local Gibbs measure $d\mathbb E_\theta^\circ\equiv\big(\int e^{-U_\theta^\circ}dx_{\Lambda_{n+1}}\big)^{-1}e^{-U_\theta^\circ}dx_{\Lambda_{n+1}}$. Then one has the following integral representation
\[\mathbb E_{\Lambda_{n+1}}^\circ|f|^2-\mathbb E_{\Lambda_n}^\circ|f|^2=\int_0^1d\theta\frac d{d\theta}\mathbb E_\theta^\circ|f|^2,\]
where the $\theta$-derivative follows $\frac d{d\theta}\mathbb E_\theta^\circ|f|^2=-\mathbb E_\theta^\circ\big(|f|^2;\frac d{d\theta}U_\theta^\circ\big)$. When the lattice is one-dimensional, the size $|\Lambda_{n+1}\setminus\Lambda_n|$ is independent of $n$ and therefore the density $d\mathbb E_{\Lambda_{n+1}}^\circ/d\mathbb E_\theta^\circ$ is bounded uniformly in $n$ and $\theta$. Then by the arguments of \cite[Lemma~5.1]{HS}, if $\mathbb E_{\Lambda_{n+1}}^\circ$ satisfies a uniform LSI the measure $\mathbb E_\theta^\circ$ also satisfies a LSI uniformly in $n$ and $\theta$. Thus, Theorem \ref{thm:Mixing} holds for $\mathbb E_\theta^\circ$ and one has the mixing condition 
\[\big|\mathbb E_\theta^\circ\big(|f|^2;\tfrac d{d\theta}U_\theta^\circ\big)\big|\leq C_{\rm mix}'|\Lambda_\theta||\Lambda_f|e^{-M'{\rm dist}(\Lambda_\theta,\Lambda_f)}\Big(\mathbb E_\theta^\circ\big|\nabla_{\Lambda_{n+1}}|f|^2\big|^2\Big)^\frac12\Big(\mathbb E_\theta^\circ\big|\nabla_{\Lambda_{n+1}}\big(\tfrac d{d\theta}U_\theta^\circ\big)\big|^2\Big)^\frac12\]
for some $C_{\rm mix}',M'>0$ independent of $n$ and $\theta$, where $\Lambda_\theta$ is the set where $\frac d{d\theta}U_\theta^\circ$ is localised in. By the construction of $\Lambda_n$, we have $|\Lambda_\theta||\Lambda_f|\leq|\Lambda_n|^2$ and ${\rm dist}(\Lambda_\theta,\Lambda_f)\geq\frac14{\rm diam}(\Lambda_n)$. Moreover, the gradient on $\frac d{d\theta}U_\theta^\circ$ is bounded by
\[\norm{\nabla_{\Lambda_{n+1}}\big(\tfrac d{d\theta}U_\theta^\circ\big)}_\infty\leq\sup_{n\in\mathbb N}\bigg(\text{\scalebox{0.8}{\raisebox{5pt}{$\sum_{\substack{X\cap(\Lambda_{n+1}\setminus\Lambda_n)\neq\varnothing\\|X|\geq2}}$}}\scalebox{0.9}{\raisebox{2pt}{$\|\nabla J_X\|_\infty$}}}\bigg)\equiv B_3.\]
Finally, by noticing that $\mathbb E_\theta^\circ(g)\leq B_1\mu(g)$ for all $\theta\in[0,1]$ and $g\geq0$, one gets the estimate
\eq{\Big|\nu_{n+1}|f|^2-\nu_n|f|^2\Big|\leq D_3|\Lambda_n|^2e^{-\varepsilon_3{\rm dist}(\Lambda_n)}\big(\mu|f|^2+\mu|\nabla f|^2\big) \label{eqn:nufdifference}}
with $D_3\equiv B_1B_3C_{\rm mix}'$ and $\varepsilon_3\equiv\frac14M'$. Gathering \eqref{eqn:nucovbound} and \eqref{eqn:nufdifference} in \eqref{muphidifference}, we have the bound
\[\Big|\mu |\phi_{n+1}|^2 - \mu |\phi_{n}|^2\Big|\leq D|\Lambda_n|^4e^{-\varepsilon{\rm diam}(\Lambda_n)}\big(\mu|f|^2+e^{2|\lambda|t}\mu|\nabla f|^2\big)\]
for some $D,\varepsilon>0$ independent of $n$. Since the series $\sum_n|\Lambda_n|^4e^{-\varepsilon{\rm diam}(\Lambda_n)}$ converges, the real sequence $\mu|\phi_n|^2$ is Cauchy and thus convergent to some $C_{\rm lim}\geq0$. Using the same argument with triangle inequality, one can show that $\big|\mu|\phi_n(t_1)|^2-\mu|\phi_n(t_2)|^2\big|$ tends to zero for any time $t_1,t_2$. Hence, the limit $C_{\rm lim}$ is independent of time and therefore identical to $\mu|f|^2$.

Using consideration of Section \ref{LSEBound} and in particular Theorem \ref{Conservation} and \ref{GradEntbound}, in the situations described by\eqref{Econv} and \eqref{Entbound} we have a compactness of the set of the solutions of \eqref{LSEcirc} associated to the sequence $\Lambda_n$ invading $\mathbb Z$. In case when entropy term is bounded we get that the sequence of solutions $\phi_n$ is weakly compact 
and since the corresponding sequence of the norms 
$\mu|\phi_n|^2$ is convergent, by general principles this implies that in fact our sequence converges strongly in $L^2(\mu)$. Hence by general principles there exists a subsequence $\phi_{n_m}$, $m\in\mathbb{N}$, which converges $\mu$-a.e. to some function $\phi\in L^2(\mu)$.
Hence we can conclude that $\phi$ satisfies the infinite-volume LSE in the ultra weak sense, i.e. for any smooth and compactly supported local  functions $g:\Omega\to\mathbb C$ one has
\[
\mu\left(g i\partial_t \phi\right) = \mu\left((-\mathcal Lg) \phi +
g\lambda \phi\log\frac{|\phi|^2}{\mu|\phi|^2}\right) 
\]
where $\mathcal L\equiv\lim_{\Lambda\to\mathbb Z}\mathcal L_\Lambda$ is well-defined for all $g$. Since the theorems in Section \ref{LSEBound} provide us with uniform bound in the corresponding $H_1(\mu)$ space we could choose a suitable weakly convergent subsequence (for which also the nonlinear term would converge as we have also uniform bound for entropy). The limit of such the sequence would satisfy the following weak form of the equation
\[
\mu\left(g i\partial_t \phi\right) = \mu\left(\nabla g\cdot\nabla \phi +
g\lambda \phi\log\frac{|\phi|^2}{\mu|\phi|^2}\right). 
\]
It is possible that following the ideas of \cite{KOZ} one could develop a technique for higher order estimates which would allow us to get a strong convergence of approximating sequence corresponding to a smooth bounded local initial data and reasonable interactions (including Gaussian perturbed by smooth short range multi-particle interactions). 
\end{proof}

In the case of the bilinear interaction, we give an explicit characterisation of the bounded multi-particle interaction condition.

\begin{exmp}
For any $X\subset\subset\mathbb Z$, consider the following perturbed quadratic potential
\[J_X(x)=\left\{\begin{tabular}{ c l }
$C_jx_j^2$&\ \ {\rm if\ }$X=\{j\}$\\
$\varepsilon W_X(x)$&\ \ {\rm if\ }$|X|\geq2$  
\end{tabular}\right.\]
where $\varepsilon\in\mathbb R$, $W_X\in C^2(\mathbb R^X)$ has uniformly in $X$ bounded derivatives of order $n=0,1,2$ and $W_X\equiv 0$ if diam$(X)>R$. Suppose $C^-\leq C_j\leq C^+$ for some $0<C^-\leq C^+$, then for sufficiently small $|\varepsilon|$ the measures $\mathbb E_\Lambda$ and $\mathbb E_\Lambda^\circ$ both satisfy a uniform LSI by Remark \ref{rmk:BE}. Moreover, by \cite[Proposition~1.3]{CFexistence}, the Cauchy problem \eqref{LSEcirc} associated to the potential $\{J_X\}$ admits a unique solution for all $\lambda\in\mathbb R$ and localised initial condition $f\in H^1(\mu)$. Therefore, all conditions in Theorem \ref{existence} are satisfied and there exists a solution of the infinite-volume LSE associated to the perturbed quadratic potential on $\mathbb R^{\mathbb Z}$.  
\end{exmp}

\begin{remark}\label{rmk:1dLattice}
    The restriction to $d=1$ is required for the finite quantity $B_1$, \eqref{densityBound}, since the interaction-range, $R$, extension of the finite cubes contain a fixed number of points. In higher dimensions, the size of the extensions grow as the cubes get larger. This could be offset with further assumptions on the potential.
\end{remark}

\section{Solitons}\label{Soliton}

In finite dimensions the solitons for the logarithmic 
Schr{\"o}dinger equation, introduced originally in   \cite{BBM3}, were discussed in number of works see e.g.   \cite{A1}, \cite{A2}, \cite{BH}, \cite{BSSSSCh}, \cite{CF}, \cite{Co}, \cite{CoC}, \cite{CoMM}, \cite{ES}, \cite{F3}, \cite{GSS1}, \cite{HAL}, \cite{HR}, \cite{KEB}, \cite{LeC}, \cite{MM}, \cite{MMT1}, \cite{MMT2}, \cite{Sc}, \cite{ZZ}.
In $\mathbb R^n$ one can try construct general solutions to the following LSE
\[i\partial_t\psi=-\Delta\psi+V\psi+\lambda\psi\log|\psi|^2\]
by performing separation of variables so that $\psi(t,x)=e^{-iEt}\phi(x)$, where $\phi:\mathbb R^n\to\mathbb R$ is the stationary state satisfying the eigen-equation
\eq{-\Delta\psi+V\psi+\lambda\psi\log|\psi|^2=E\psi \label{eigeneq}}
with energy eigenvalue $E\in\mathbb R$. For the sake of simplicity, we shall first consider $V=0$ and try to solve 
\eq{-\Delta\psi+\lambda\psi\log|\psi|^2=E\psi. \label{simpleeigen}} 
Inspired by the results in \cite{BBM1}, we take the soliton-like ansatz $\psi=\exp(-\frac a2|x|^2)$ into \eqref{simpleeigen} and find that $a=-\lambda$ provides a solution with $E=-\lambda n$. Letting $\psi_0=\exp(-\frac\lambda2|x|^2)$, one can readily see that for any nonzero $b\in\mathbb R$, $\psi_b\equiv b\psi_0$ is also a solution to \eqref{simpleeigen} with eigenvalue $E_b=-\lambda(n-\log b^2)$. In order $\psi$ to be normalizable, we require $\lambda<0$ and the normalised solution is $\psi=(-2\pi/\lambda)^\frac n2\exp(-\frac\lambda2|x|^2)$ with eigenvalue $E=-\lambda n(1-\log(-2\pi/\lambda))$. Next, we investigate \eqref{eigeneq} with a quadratic potential $V=\text{\raisebox{0.1pt}{\scalebox{0.9}{$\sum_{j=1}^n$}}}a_jx_j^2$ where $a_j>0$, and this time we use a more general ansatz $\psi=\exp(-U(x))$ to get 
\[\Delta U-|\nabla U|^2-2\lambda U=E-\sum_{j=1}^na_jx_j^2.\]
Choosing $U=\text{\raisebox{0.3pt}{\scalebox{0.9}{$\sum_{j=1}^n$}}}\frac12c_jx_j^2$, with $c_j=\frac12(\sqrt{4a_j+\lambda^2}-\lambda)>0$, one gets a normalised solution   $\psi=\text{\raisebox{0.4pt}{\scalebox{0.8}{$\sqrt{(2\pi)^n/\text{\raisebox{1pt}{$\scriptstyle\prod$}}_jc_j}$}}}\exp(-\text{\raisebox{0.7pt}{$\scriptstyle\sum$}}_jc_jx_j^2)$ with eigenvalue $E=\lambda n\log(2\pi)+\text{\raisebox{0.7pt}{$\scriptstyle\sum$}}_j(c_j-\lambda\log c_j)$.

Use of Galilean covariance provides us with a soliton type solution \cite{BH}. However for more general potential this transformation requires suitable adaptation of the potential.

In the case of composite systems consisting of a number of elementary entities (atoms, molecules, cells, etc) one would also like to model states where the wave function changes in time, with concentration to a given part of the system. 

In infinite dimensional systems with nontrivial interaction the typical situation is more complicated.
Under log-Sobolev inequality, (thanks to Lemma \ref{LSItoSGI}), assuming $\lambda > -\frac{\varepsilon}{c}$ with some $\varepsilon\in[0,1)$, we have the following analog of spectral gap 
property
\[
\mu\left(  -\phi^\ast \mathcal L\phi+\lambda  |\phi|^2\log\frac{|\phi|^2}{\mu |\phi|^2} \right) \geq \frac{2(1-\varepsilon)}{c} \left(\frac{1}{c}+\lambda\right)\mu|\phi|^2
.\]
However, in general the structure of spectrum is not discrete. In the best case what one encounters (for example in
statistical mechanics see e.g.
\cite{MT}, \cite{MT} or QFT see e.g. \cite{GJ} and references there in) are separated windows of continuous spectrum. Thus the classical notion of soliton may not make sense.
However it is possible that states with soliton behaviour on some large timescale could be reasonable. It would be interesting to develop such theory (with potential applications e.g. to quantum computing).

\bibliographystyle{ieeetr}
\bibliography{Reference.bib}

\begin{thebibliography}{10}

\bibitem{BBM1}
I.~Bialynicki-Birula and J.~Mycielski, ``Nonlinear wave mechanics,'' {\em
  Annals of Physics}, vol.~100, no.~1-2, pp.~62--93, 1976.

\bibitem{BBM2}
I.~Bia{\l}ynicki-Birula and J.~Mycielski, ``Uncertainty relations for
  information entropy in wave mechanics,'' {\em Communications in Mathematical
  Physics}, vol.~44, no.~2, pp.~129--132, 1975.

\bibitem{BSSSSCh}
H.~Buljan, A.~{\v{S}}iber, M.~Solja{\v{c}}i{\'c}, T.~Schwartz, M.~Segev, and
  D.~N. Christodoulides, ``Incoherent white light solitons in logarithmically
  saturable noninstantaneous nonlinear media,'' {\em Physical Review E},
  vol.~68, no.~3, p.~036607, 2003.

\bibitem{DeMFGL}
S.~De~Martino, M.~Falanga, C.~Godano, and G.~Lauro, ``Logarithmic
  schr{\"o}dinger-like equation as a model for magma transport,'' {\em EPL
  (Europhysics Letters)}, vol.~63, no.~3, p.~472, 2003.

\bibitem{H}
E.~F. Hefter, ``Application of the nonlinear schr{\"o}dinger equation with a
  logarithmic inhomogeneous term to nuclear physics,'' {\em Physical Review A},
  vol.~32, no.~2, p.~1201, 1985.

\bibitem{HAL}
T.~Hansson, D.~Anderson, and M.~Lisak, ``Propagation of partially coherent
  solitons in saturable logarithmic media: A comparative analysis,'' {\em
  Physical Review A}, vol.~80, no.~3, p.~033819, 2009.

\bibitem{Br}
J.~D. Brasher, ``Nonlinear wave mechanics, information theory, and
  thermodynamics,'' {\em International journal of theoretical physics},
  vol.~30, no.~7, pp.~979--984, 1991.

\bibitem{Y}
K.~Yasue, ``Quantum mechanics of nonconservative systems,'' {\em Annals of
  Physics}, vol.~114, no.~1, pp.~479--496, 1978.

\bibitem{Zl}
K.~G. Zloshchastiev, ``Logarithmic nonlinearity in theories of quantum gravity:
  Origin of time and observational consequences,'' {\em Gravitation and
  Cosmology}, vol.~16, pp.~288--297, Oct 2010.

\bibitem{AZ}
A.~V. Avdeenkov and K.~G. Zloshchastiev, ``Quantum bose liquids with
  logarithmic nonlinearity: Self-sustainability and emergence of spatial
  extent,'' {\em Journal of Physics B: Atomic, Molecular and Optical Physics},
  vol.~44, no.~19, p.~195303, 2011.

\bibitem{A1}
A.~H. Ardila, ``Existence and stability of standing waves for nonlinear
  fractional schr{\"o}dinger equation with logarithmic nonlinearity,'' {\em
  Nonlinear Analysis}, vol.~155, pp.~52--64, 2017.

\bibitem{A2}
A.~H. Ardila, ``Orbital stability of gausson solutions to logarithmic
  schr{\"o}dinger equations,'' {\em Electron. J. Diff. Eqns.}, vol.~2016 (335),
  pp.~1--9, 2016.

\bibitem{BCST}
W.~Bao, R.~Carles, C.~Su, and Q.~Tang, ``Error estimates of a regularized
  finite difference method for the logarithmic schrödinger equation,'' {\em
  SIAM Journal on Numerical Analysis}, vol.~57, no.~2, pp.~657--680, 2019.

\bibitem{CG}
R.~Carles and I.~Gallagher, ``Universal dynamics for the defocusing logarithmic
  schr{\"o}dinger equation,'' {\em Duke Mathematical Journal}, vol.~167, no.~9,
  pp.~1761--1801, 2018.

\bibitem{CS}
R.~Carles and C.~Su, ``Nonuniqueness and nonlinear instability of gaussons
  under repulsive harmonic potential,'' {\em arXiv preprint arXiv:2107.10024},
  2021.

\bibitem{C1}
T.~Cazenave, ``Stable solutions of the logarithmic schr{\"o}dinger equation,''
  {\em Nonlinear Analysis: Theory, Methods \& Applications}, vol.~7, no.~10,
  pp.~1127--1140, 1983.

\bibitem{CH1}
T.~Cazenave and A.~Haraux, ``{\'E}quations d'{\'e}volution avec non
  lin{\'e}arit{\'e} logarithmique,'' in {\em Annales de la Facult{\'e} des
  sciences de Toulouse: Math{\'e}matiques}, vol.~2, pp.~21--51, 1980.

\bibitem{CH2}
T.~Cazenave, {\em Semilinear Schrodinger Equations}, vol.~10.
\newblock American Mathematical Soc., 2003.

\bibitem{DAMS}
P.~d'Avenia, E.~Montefusco, and M.~Squassina, ``On the logarithmic
  schr{\"o}dinger equation,'' {\em Communications in Contemporary Mathematics},
  vol.~16, no.~02, p.~1350032, 2014.

\bibitem{F1}
G.~Ferriere, ``The focusing logarithmic schrödinger equation: Analysis of
  breathers and nonlinear superposition,'' {\em Discrete $\&$ Continuous
  Dynamical Systems}, vol.~40, no.~11, pp.~6247--6274, 2020.

\bibitem{F2}
G.~Ferriere, ``Convergence rate in wasserstein distance and semiclassical limit
  for the defocusing logarithmic schr{\"o}dinger equation,'' {\em Analysis \&
  PDE}, vol.~14, no.~2, pp.~617--666, 2021.

\bibitem{GLN}
P.~Guerrero, J.~L{\'o}pez, and J.~Nieto, ``Global h1 solvability of the 3d
  logarithmic schr{\"o}dinger equation,'' {\em Nonlinear Analysis: Real World
  Applications}, vol.~11, no.~1, pp.~79--87, 2010.

\bibitem{HMN}
H.-M. Nguyen and M.~Squassina, ``Logarithmic sobolev inequality revisited,''
  {\em Comptes Rendus Mathematique}, vol.~355, no.~4, pp.~447--451, 2017.

\bibitem{T}
W.~C. Troy, ``Uniqueness of positive ground state solutions of the logarithmic
  schr{\"o}dinger equation,'' {\em Archive for Rational Mechanics and
  Analysis}, vol.~222, no.~3, pp.~1581--1600, 2016.

\bibitem{GZ}
A.~Guionnet and B.~Zegarlinksi, ``Lectures on logarithmic sobolev
  inequalities,'' in {\em S{\'e}minaire de probabilit{\'e}s XXXVI}, pp.~1--134,
  Springer, 2003.

\bibitem{SZ1}
D.~W. Stroock and B.~Zegarlinski, ``The logarithmic sobolev inequality for
  continuous spin systems on a lattice,'' {\em Journal of functional analysis},
  vol.~104, no.~2, pp.~299--326, 1992.

\bibitem{Z1}
B.~Zegarlinski, ``Log-sobolev inequalities for infinite one dimensional lattice
  systems,'' {\em Communications in Mathematical Physics}, vol.~133,
  pp.~147--162, Sep 1990.

\bibitem{Z2}
B.~Zegarlinski, ``Dobrushin uniqueness theorem and logarithmic sobolev
  inequalities,'' {\em Journal of Functional Analysis}, vol.~105, no.~1,
  pp.~77--111, 1992.

\bibitem{Z3}
B.~Zegarlinski, ``The strong decay to equilibrium for the stochastic dynamics
  of unbounded spin systems on a lattice,'' {\em Communications in Mathematical
  Physics}, vol.~175, pp.~401--432, Jan 1996.

\bibitem{SZ2}
D.~W. Stroock and B.~Zegarlinski, ``The logarithmic sobolev inequality for
  discrete spin systems on a lattice,'' {\em Communications in Mathematical
  Physics}, vol.~149, no.~1, pp.~175--193, 1992.

\bibitem{BHK}
J.~B{\'e}llissard and R.~H{\o}egh-Krohn, ``Compactness and the maximal gibbs
  state for random gibbs fields on a lattice,'' {\em Communications in
  Mathematical Physics}, vol.~84, no.~3, pp.~297--327, 1982.

\bibitem{DLR2}
C.~Preston, ``Random fields and specifications,'' in {\em Random Fields},
  pp.~11--32, Springer, 1976.

\bibitem{DLR1}
H.-O. Georgii, ``Gibbs measures and phase transitions,'' {\em de Gruyter
  Studies in Mathematics}, vol.~9, 1988.

\bibitem{DLR3}
Y.~G. Sinai, {\em Theory of phase transitions: rigorous results}.
\newblock Pergamon, 1982.

\bibitem{FV}
S.~Friedli and Y.~Velenik, {\em Statistical mechanics of lattice systems: a
  concrete mathematical introduction}.
\newblock Cambridge University Press, 2017.

\bibitem{Rothaus1}
O.~S. Rothaus, ``Logarithmic sobolev inequalities and the spectrum of
  schr{\"o}dinger operators,'' {\em Journal of functional analysis}, vol.~42,
  no.~1, pp.~110--120, 1981.

\bibitem{Rothaus2}
O.~S. Rothaus, ``Lower bounds for eigenvalues of regular sturm-liouville
  operators and the logarithmic sobolev inequality,'' {\em Duke Mathematical
  Journal}, vol.~45, no.~2, pp.~351--362, 1978.

\bibitem{BE}
D.~Bakry and M.~{\'E}mery, ``Diffusions hypercontractives,'' in {\em Seminaire
  de probabilit{\'e}s XIX 1983/84}, pp.~177--206, Springer, 1985.

\bibitem{Ing}
J.~D. Inglis, {\em Coercive inequalities for generators of H{\"o}rmander type}.
\newblock PhD thesis, Department of Mathematics, Imperial College London, 2010.

\bibitem{HM}
C.~Henderson and G.~Menz, ``Equivalence of a mixing condition and the lsi in
  spin systems with infinite range interaction,'' {\em Stochastic Processes and
  their Applications}, vol.~126, no.~10, pp.~2877--2912, 2016.

\bibitem{HeZeg}
W.~Hebisch and B.~Zegarlinski, ``Coercive inequalities on metric measure
  spaces,'' {\em Journal of Functional Analysis}, vol.~258, no.~3,
  pp.~814--851, 2010.

\bibitem{SZ3}
D.~W. Stroock and B.~Zegarlinski, ``The equivalence of the logarithmic sobolev
  inequality and the dobrushin-shlosman mixing condition,'' {\em Communications
  in mathematical physics}, vol.~144, no.~2, pp.~303--323, 1992.

\bibitem{YJY}
H.~Ye, J.~Gao, and Y.~Ding, ``A generalized gronwall inequality and its
  application to a fractional differential equation,'' {\em Journal of
  Mathematical Analysis and Applications}, vol.~328, no.~2, pp.~1075--1081,
  2007.

\bibitem{HS}
R.~Holley and D.~Stroock, ``Logarithmic sobolev inequalities and stochastic
  ising models,'' {\em Journal of Statistical Physics}, vol.~46, no.~5,
  pp.~1159--1194, 1987.

\bibitem{KOZ}
V.~Kontis, M.~Ottobre, and B.~Zegarlinski, ``Long-and short-time behaviour of
  hypocoercive-type operators in infinite dimensions: An analytic approach,''
  {\em Infinite Dimensional Analysis, Quantum Probability and Related Topics},
  vol.~20, no.~03, p.~1750015, 2017.

\bibitem{CFexistence}
R.~Carles and G.~Ferriere, ``Logarithmic schr{\"o}dinger equation with
  quadratic potential,'' {\em Nonlinearity}, vol.~34, no.~12, p.~8283, 2021.

\bibitem{BBM3}
I.~Bialynicki-Birula and J.~Mycielski, ``Gaussons: solitons of the logarithmic
  schr{\"o}dinger equation,'' {\em Physica Scripta}, vol.~20, no.~3-4, p.~539,
  1979.

\bibitem{BH}
H.~R. Brown and P.~R. Holland, ``The galilean covariance of quantum mechanics
  in the case of external fields,'' {\em American Journal of Physics}, vol.~67,
  no.~3, pp.~204--214, 1999.

\bibitem{CF}
R.~C{\^o}te and X.~Friederich, ``On smoothness and uniqueness of multi-solitons
  of the non-linear schr{\"o}dinger equations,'' {\em Communications in Partial
  Differential Equations}, vol.~46, no.~12, pp.~2325--2385, 2021.

\bibitem{Co}
R.~C{\^o}te, ``On the soliton resolution for equivariant wave maps to the
  sphere,'' {\em Communications on Pure and Applied Mathematics}, vol.~68,
  no.~11, pp.~1946--2004, 2015.

\bibitem{CoC}
R.~C{\^o}te and S.~Le~Coz, ``High-speed excited multi-solitons in nonlinear
  schr{\"o}dinger equations,'' {\em Journal de math{\'e}matiques pures et
  appliqu{\'e}es}, vol.~96, no.~2, pp.~135--166, 2011.

\bibitem{CoMM}
R.~C{\^o}te, Y.~Martel, and F.~Merle, ``Construction of multi-soliton solutions
  for the $ l^2$-supercritical gkdv and nls equations,'' {\em Revista
  Matematica Iberoamericana}, vol.~27, no.~1, pp.~273--302, 2011.

\bibitem{ES}
W.~Eckhaus and P.~Schuur, ``The emergence of solitons of the korteweg-de vries
  equation from arbitrary initial conditions,'' {\em Mathematical Methods in
  the Applied Sciences}, vol.~5, no.~1, pp.~97--116, 1983.

\bibitem{F3}
G.~Ferriere, ``Existence of multi-solitons for the focusing logarithmic
  non-linear schr{\"o}dinger equation,'' in {\em Annales de l'Institut Henri
  Poincar{\'e} C, Analyse non lin{\'e}aire}, vol.~38, pp.~841--875, Elsevier,
  2021.

\bibitem{GSS1}
M.~Grillakis, J.~Shatah, and W.~Strauss, ``Stability theory of solitary waves
  in the presence of symmetry, i,'' {\em Journal of functional analysis},
  vol.~74, no.~1, pp.~160--197, 1987.

\bibitem{HR}
E.~Hernandez and B.~Remaud, ``General properties of gausson-conserving
  descriptions of quantal damped motion,'' {\em Physica A: Statistical
  Mechanics and its Applications}, vol.~105, no.~1-2, pp.~130--146, 1981.

\bibitem{KEB}
W.~Kr{\'o}likowski, D.~Edmundson, and O.~Bang, ``Unified model for partially
  coherent solitons in logarithmically nonlinear media,'' {\em Physical Review
  E}, vol.~61, no.~3, p.~3122, 2000.

\bibitem{LeC}
S.~Le~Coz, ``Standing waves in nonlinear schr{\"o}dinger equations,'' {\em
  Analytical and numerical aspects of partial differential equations},
  pp.~151--192, 2009.

\bibitem{MM}
Y.~Martel and F.~Merle, ``Multi solitary waves for nonlinear schr{\"o}dinger
  equations,'' {\em Annales de l'Institut Henri Poincar{\'e} C}, vol.~23,
  no.~6, pp.~849--864, 2006.

\bibitem{MMT1}
Y.~Martel, F.~Merle, and T.-P. Tsai, ``Stability and asymptotic stability for
  subcritical gkdv equations,'' {\em Communications in mathematical physics},
  vol.~231, no.~2, pp.~347--373, 2002.

\bibitem{MMT2}
Y.~Martel, F.~Merle, and T.-P. Tsai, ``Stability in h1 of the sum of k solitary
  waves for some nonlinear schr{\"o}dinger equations,'' {\em Duke Mathematical
  Journal}, vol.~133, no.~3, pp.~405--466, 2006.

\bibitem{Sc}
P.~C. Schuur, {\em Asymptotic analysis of soliton problems: an inverse
  scattering approach}, vol.~1232.
\newblock Springer, 2006.

\bibitem{ZZ}
C.~Zhang and X.~Zhang, ``Bound states for logarithmic schr{\"o}dinger equations
  with potentials unbounded below,'' {\em Calculus of Variations and Partial
  Differential Equations}, vol.~59, p.~23, Jan 2020.

\bibitem{MT}
R.~A. Minlos and A.~Trishch, ``The complete spectral decomposition of a
  generator of glauber dynamics for the one-dimensional ising model,'' {\em
  Russian Mathematical Surveys}, vol.~49, no.~6, p.~210, 1994.

\bibitem{GJ}
J.~Glimm and A.~Jaffe, {\em Quantum physics: a functional integral point of
  view}.
\newblock Springer, 1987.

\end{thebibliography}

\end{document}